\documentclass[11pt]{article}%
\usepackage{amsmath,latexsym}
\usepackage{amssymb}
\usepackage{graphicx}
\usepackage{bbm} 
\usepackage{subfigure}
\newcommand{\fine}{$\Box$}		    

\newtheorem{proposition}{\bf Proposition}
\newtheorem{corollary}{\bf Corollary}
 
\newenvironment{proof}{
\begin{trivlist}
\item[\hspace{\labelsep}{\bf\noindent Proof. }] }{\par\hfill\end{trivlist}
\par}
\usepackage{mathrsfs} 
\input{epsf.sty}
\title{\huge\bf
$M/M/1$ queue in two alternating environments and its heavy traffic approximation
\date{Author's version.  Published in:  {\em Journal of Mathematical Analysis and Applications}\ 465 (2018), pp.\ 973-1001,   
doi:   10.1016/j.jmaa.2018.05.043 \ -- \ 
URL: https://www.sciencedirect.com/sdfe/reader/pii/S0022247X18304360/pdf}
}

\author{
\large \bf Antonio Di Crescenzo\footnote{
Dipartimento di Matematica, Universit\`a degli Studi di Salerno, 
Via Giovanni Paolo II n.132, 84084 Fisciano (SA), Italy, 
E-mail: adicrescenzo@unisa.it}
\qquad
Virginia Giorno\footnote{
Dipartimento di Informatica,
Universit\`a di Salerno, Via Giovanni Paolo II n.\ 132, 84084 Fisciano (SA), Italy.
E-mail: giorno@unisa.it}
\\
\large \bf  Balasubramanian Krishna Kumar\footnote{
Department of Mathematics, Anna University, Chennai 600 025, India. 
E-mail: drbkkumar@hotmail.com}
\qquad
Amelia G. Nobile\footnote{
Dipartimento di Informatica,
Universit\`a di Salerno, Via Giovanni Paolo II n.\ 132, 84084 Fisciano (SA), Italy.
E-mail: nobile@unisa.it}
}

\begin{document} 
 
\maketitle

\begin{abstract}
We  investigate an $M/M/1$ queue operating in two switching environments, where the switch is 
governed by a two-state time-homogeneous Markov chain. This model allows to describe a system 
that is subject to regular operating phases alternating with anomalous working phases or random 
repairing periods. We first obtain the steady-state distribution of the process in terms of a generalized 
mixture of two geometric distributions.  In the special case when only one kind of switch is allowed, 
we analyze the transient distribution, and investigate the busy period problem. The analysis is also 
performed by means of a suitable heavy-traffic approximation which leads to a continuous random 
process. Its distribution satisfies a partial differential equation with randomly alternating infinitesimal 
moments. For the approximating process we determine the steady-state distribution, the transient 
distribution and a first-passage-time density.

\end{abstract}

\section{Introduction}\label{section1}
The $M/M/1$ queue is the most well-known queueing system, whose customers arrive according 
to a Poisson process, and the service times are exponentially distributed. Its generalizations are often 
employed to describe more complex systems, such as queues in the presence of catastrophes 
(see, for instance, Di Crescenzo {\em et al.}~\cite{DGN2003}, Kim and Lee \cite{Kimetal2014}, and 
Krishna Kumar and Pavai Madheswari~\cite{KP2005}). In some cases, the sequence of repeated 
catastrophes and successive repairs yields  alternating operative phases (see, for instance, 
Paz and Yechiali~\cite{PaYe2014} and Jiang {\em et al.}~\cite{Jiangetal2015} for the analysis of 
queues in a multi-phase random environment). Moreover, realistic situations related to queueing 
services are often governed by state-dependent rates (cf., for instance, Giorno {\em et al.}~\cite{GNP2018}), 
or by alternating behavior, such as cyclic polling systems (cf.\ Avissar and Yechiali~\cite{AvYe2012}). 
Specifically, the analysis of queueing systems characterized by alternating mechanisms has been 
object of investigation largely in the past. The first systematic contribution in this area was provided 
in Yechiali and Naor \cite{YeNa1971}, where the $M/M/1$ queue was analyzed in the steady-state 
regime when the rates of arrival and service are subject to Poisson alternations. 
A recent study due Huang and Lee \cite{HL2013} is concerning a similar queueing model with 
a finite size queue and a service mechanism characterized by randomly alternating behavior. 
\par
Other types of complex systems in alternating environment are provided by two single server 
queues, where customers arrive in a single stream, and each arrival creates simultaneously 
the work demands to be served by the two servers. Instances of such two-queue polling models have 
been studied in Boxma {\em et al.}~\cite{BoScYe2002} and Eliazar {\em et al.}~\cite{ElFiYe2002}. 
In other cases, instead, the alternating behavior of queueing systems is described by time-dependent 
arrival and service rates, such as in the $M_t/M_t/1$ queue subject to under-, over-, and critical loading. 
Heavy-traffic diffusion approximations or asymptotic expansions for such types of systems have 
been investigated in Di Crescenzo and Nobile~\cite{DCNo95}, Giorno {\em et al.}~\cite{GNR87}, 
Mandelbaum and Massey~\cite{MaMa1995}. 
\par
Attention has been devoted in the literature also to queues with more complex random switching 
mechanisms, that arise naturally in the study of packet arrivals to a local switch (see, for instance, 
Burman and Smith~\cite{BurSmi1986}). Similar mechanisms have been studied recently also by 
Arunachalam {\em et al.}~\cite{ArGuDh2010}, Pang and  Zhou~\cite{PaZh2016}, 
Liu and Yu~\cite{LY2016} 
and Perel and Yechiali~\cite{PeYe2017}. 
\par
Contributions in the area of queues with randomly varying arrival and service rates are due to 
Neuts~\cite{Neuts},  Kao and Lin~\cite{KaoLin}, and Lu and Serfozo~\cite{LuSerfozo}. 
Moreover, Boxma and Kurkova~\cite{BoKu2000} studied an $M/M/1$ queue for which the speed of 
the server alternates between two constant values, according to different time distributions. 
A similar problem for the $M/G/1$ queue was studied  by the same authors in \cite{BoKu2001}, 
whereas the case of the $M/M/\infty$ queue is treated by D'Auria~\cite{DAuria2014}. 

\subsection{Motivations} 
Along the lines of the above mentioned investigations, in this paper we study an $M/M/1$ 
queue subject to alternating behavior. 
The basic model retrace the alternating $M/M/1$ queue studied in \cite{YeNa1971}. 
Indeed, we assume that the characteristics of the queue are fluctuating randomly in time, under two 
operating environments which alternate randomly. Initially the system starts under the first 
environment with probability $p$, or from the second one w.p.\ $1-p$. Then, 
at time $t$ the customers arrival rates and the service rates are $(\lambda_i,\mu_i)$ 
if the  operational environment  is $\mathscr{E}(t)=i$, for $i=1,2$. 
The operational environment switches from $\mathscr{E}(t)=1$ to $\mathscr{E}(t)=2$ with 
rate $\eta_1$, whereas the reverse switch occurs with rate $\eta_2$. This setting allows 
to model queues based on two modes of customers arrivals, with fluctuating high-low 
rates, where the service rate is instantaneously adapted to the new arrival conditions. 
Moreover, the considered 
model is also suitable to describe instances in which only 
one kind of rate is subject to random fluctuations. For instance, the case when only the 
service rate is alternating between two values $\mu_1$ and $\mu_2$ refers to a  
queue which is subject to randomly occurring catastrophes, whose effect is to transfer 
the service mechanism to a slower server for the duration of a random repair time. 
\par
The above stated assumptions are also paradigmatic of realistic situations in which the 
underlying mechanism of the queue is affected by external conditions that 
alternate randomly, such as  systems subject to interruptions, or up-down periods. 
In this case, the adaptation of the rates occurs instantaneously, differently from other 
settings where customers observe the queue level before taking a decision 
(see, for instance,  Economou and Manou~\cite{EcMa2016}). 
\par
It is relevant to point out that the alternation between the rates may produce regulation effects 
for the queue mechanism. Indeed, if the current environment leads to a traffic congestion 
(i.e., $\lambda_i>\mu_i$ for $\mathscr{E}(t)=i$), then the switch to the other environment may 
yield a favorable consequence for the queue length (if $\lambda_{3-i}<\mu_{3-i}$). This can be 
achieved by increasing the service speed, or decreasing the customer arrival rates. Note that 
the above conditions on the arrival and service rates, with appropriate switching rates 
$\eta_1$ and $\eta_2$, may lead to a stable queueing system, even if the queue is not stable 
under one of the two environments. 
\subsection{Plan of the paper} 
In Section~\ref{section2} we investigate the  distribution of the number of customers and the 
current environment of the considered alternating queue. We first obtain the steady-state distribution 
of the system, 
which is expressed as a generalized mixture of two geometric distributions. This result 
provides an alternative solution to that obtained in \cite{YeNa1971} with a different approach. 
It is worth noting that the system admits of a steady-state distribution even in a case 
when one of the alternating environments does not possess a steady state. 
Furthermore, we also obtain  the conditional means and the  entropies of the process. 
\par
The transient probability distribution of the queue  is  studied in Section~\ref{section3}. 
Since the general case is not tractable, we analyse such distribution under the assumption that 
only a switch from environment $\mathscr{E}(t)=1$ to environment $\mathscr{E}(t)=2$ is allowed. 
In this case, we express the transient probabilities in a series form which involves 
the same distribution in the absence of environment switch. A similar result is also obtained 
for the first-passage-time (FPT) density through the zero state, aiming to investigate   
the busy period. The Laplace transform of the 
FPT density is also determined in order to evaluate the probability of busy period termination, 
and the related expectation. 
\par
In order to investigate the queueing system also under more general conditions we are lead to construct 
a heavy-traffic diffusion approximation of the queue-length process. This is obtained in Section~\ref{section4} 
by means of a customary scaling procedure similar to those adopted in Dharmaraja {\em et al.}\  
\cite{DDGN2015} and Di Crescenzo {\em et al.}\  \cite{DGN2003}. The distribution of the 
approximating continuous process satisfies a suitable partial differential equation with alternating terms. 
Examples of diffusive systems with alternating behavior can be found in the physics 
literature. For instance, Bez\'ak \cite{Be92} studied a modified Wiener process subject to Poisson-paced 
pulses. In this case the effect of pulses is the alternation of the infinitesimal variance. A similar (unrestricted) 
diffusion process characterized by alternating drift and constant infinitesimal variance 
has been studied in Di Crescenzo {\em et al.}\  \cite{DCDNR2005} and \cite{DCZ2015}. 
This is different from the  approximating diffusion process treated here, 
for which all infinitesimal moments are alternating. The approach adopted in \cite{DCZ2015} 
cannot be followed for the process under heavy traffic, since it is restricted by a reflecting boundary at 0. 
\par
Concerning the approximating process, which can be viewed as an
alternating Wiener process, we determine the steady-state density, expressed 
as a generalized mixture of two exponential densities. Then, in Section~\ref{section5}
for the approximating diffusion process we obtain the transient distribution when only one kind of switch is allowed. 
The distribution is decomposed in an integral form that involves the expressions of the classical Wiener process in the 
presence of a reflecting boundary at zero. Also for the alternating diffusion process we investigate the 
FPT density through the zero state, in order to come to a suitable approximation of the 
busy period. In this case, we express the related distribution in an integral form, and develop a Laplace 
transform-based approach aimed to study the FPT mean. 
\par
In the paper, the quantities of interest are investigated through computationally effective procedures by using 
MATHEMATICA$^{\footnotesize{\rm \textregistered}}$. 
%
\section{The queueing model}\label{section2}
Let $\{{\bf N}(t)=[N(t),\mathscr{E}(t)],t\geq 0\}$ be a two-dimensional continuous-time Markov chain, 
having state-space $\mathbb{N}_0\times\{1,2\}$ and transient probabilities  
\begin{equation}
 p_{n,i}(t)=\mathbb P[{\bf N}(t)=(n,i)], \qquad  n\in \mathbb{N}_0, \quad i=1,2, \quad t \geq 0,
 \label{eq:transprobab}
\end{equation}
where 
\begin{equation}
{\bf N}(0)=
\left\{\begin{array}{ll}
(j,1),&\;{\rm with\;probability}\; p,\\
(j,2),&\;{\rm with\;probability}\;1-p,
\end{array}\right.
 \label{eq:initcondit_N}
\end{equation}
with $j\in \mathbb{N}_0$. Here, $N(t)$ describes the number of customers at time $t$ in a $M/M/1$ queueing system 
operating under two randomly switching environments, and $\mathscr{E}(t)$ denotes the operational environment 
at time $t$. Specifically, if $\mathscr{E}(t)=i$ then the arrival rate of customers at time $t$ is $\lambda_i$ whereas 
the service rate is $\mu_i$, for $i=1,2$, with constant parameters $\lambda_1,\lambda_2, \mu_1,\mu_2>0$. 
\par
We assume that  two operational regimes alternate according to fixed constant rates. 
In other terms, if the system is operating at time $t$ in the environment  $\mathscr{E}(t)=1$ then it  switches to 
the environment $\mathscr{E}(t)=2$ with rate $\eta_1\geq 0$, whereas if $\mathscr{E}(t)=2$ then the system 
switches in the environment $\mathscr{E}(t)=1$ with rate $\eta_2\geq 0$, with $\eta_1+\eta_2>0$. 
Figure~\ref{fig:chain_unilateral} shows the state diagram of ${\bf N}(t)$. 
We recall that the considered setting is in agreement with the model introduced in \cite{YeNa1971}. 
%
\begin{figure}[t]  
\centering
\includegraphics[scale=0.45]{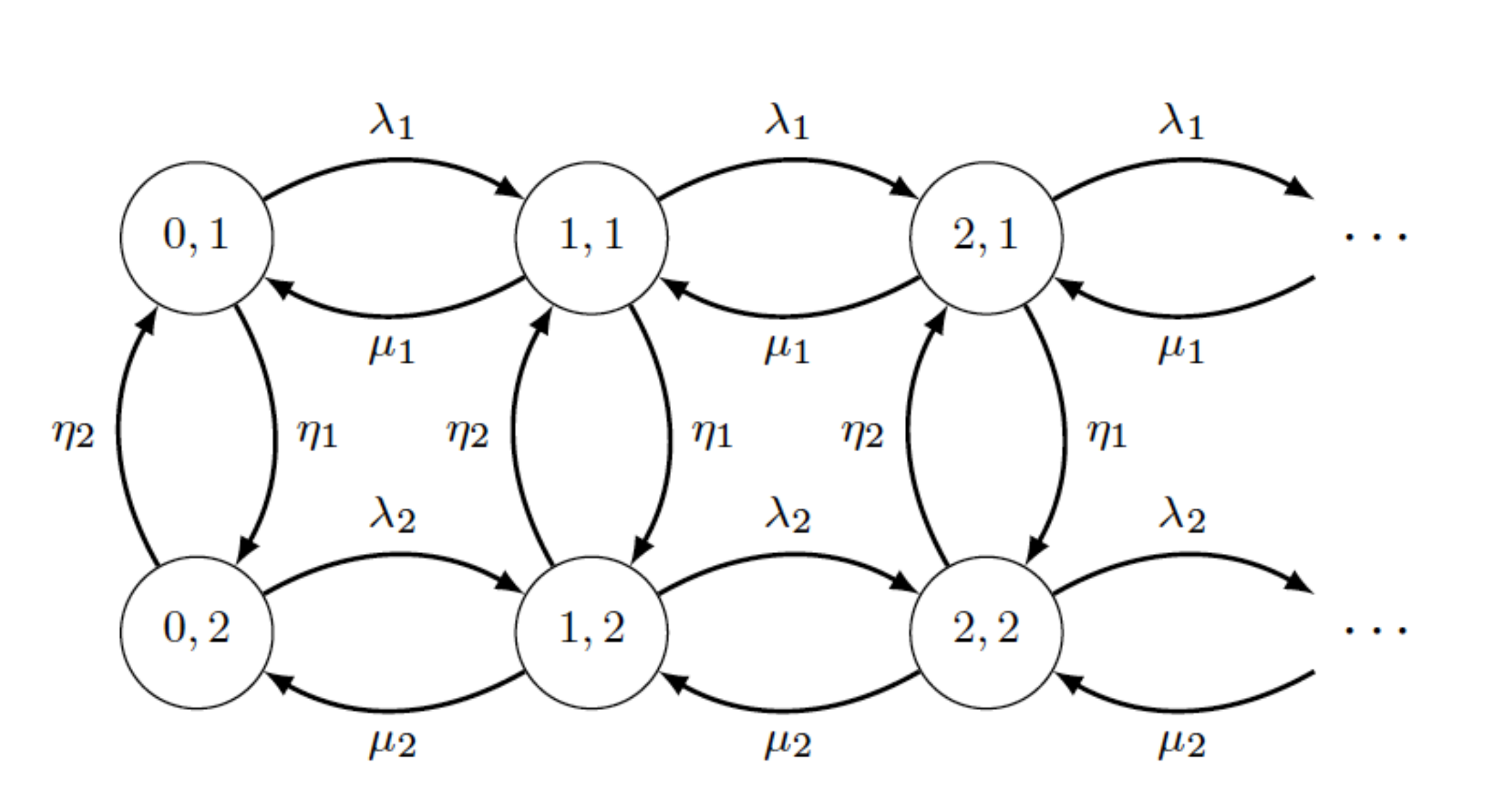}\\
\caption{The state diagram of the Markov chain ${\bf N}(t)$.}
\label{fig:chain_unilateral}
\end{figure}
\par
For a fixed $j\in \mathbb{N}_0$, we assume that the system is subject to random initial conditions given by a Bernoulli trial 
on the states $(j,1)$ and $(j,2)$. Indeed, for a given  $p\in [0,1]$, recalling (\ref{eq:initcondit_N}),  we have 
\begin{equation}
 p_{n,1}(0)=p\, \delta_{n,j}, \qquad p_{n,2}(0)=(1-p)\,\delta_{n,j},
 \label{initial_condition}
\end{equation}
where $\delta_{n,j}$ is the Kronecker's delta. 
\par
From the specified assumptions we have the following forward Kolmogorov equations for the 
first operational regime: 
\begin{eqnarray}
&&\hspace{-0.8cm} 
{dp_{0,1}(t)\over dt}=-(\lambda_1+\eta_1)\,p_{0,1}(t)+\eta_2\,p_{0,2}(t)+\mu_1\,p_{1,1}(t),
\nonumber\\
&&\hspace{-0.8cm} 
{dp_{n,1}(t)\over dt}=-(\lambda_1+\mu_1+\eta_1)\,p_{n,1}(t)+\eta_2\,p_{n,2}(t)+\mu_1\,p_{n+1,1}(t)+\lambda_1\,p_{n-1,1}(t),\nonumber\\
&& \hspace*{9cm} n\in \mathbb N, 
 \label{equat_env1} 
\end{eqnarray}
and for the second operational regime:
\begin{eqnarray}
&&\hspace{-0.8cm} 
{dp_{0,2}(t)\over dt}=-(\lambda_2+\eta_2)\,p_{0,2}(t)+\eta_1\,p_{0,1}(t)+\mu_2\,p_{1,2}(t)\nonumber\\
&&\hspace{-0.8cm} 
{dp_{n,2}(t)\over dt}=-(\lambda_2+\mu_2+\eta_2)\,p_{n,2}(t)+\eta_1\,p_{n,1}(t)+\mu_2\,p_{n+1,2}(t)+\lambda_2\,p_{n-1,2}(t),\nonumber\\
&& \hspace*{9cm} n\in \mathbb N. 
 \label{equat_env2} 
\end{eqnarray}
Clearly, for all $t\geq 0$ one has:
\begin{equation}
\sum_{n=0}^{+\infty}\bigl[p_{n,1}(t)+p_{n,2}(t)\bigr]=1.
\label{normalization_condition}
\end{equation}
%
\subsection{Steady-state distribution}
%
Let us now investigate the steady-state distribution of the two-environ\-ment $M/M/1$ queue. 
We will show that it can be expressed as a generalized mixture of two geometric distributions. 
Our approach is different from the analysis performed in \cite{YeNa1971}, where the 
steady-state distribution is achieved through recursive formulas. 
\par
Let ${\bf N}=(N,\mathscr{E})$ be the two-dimensional random variable describing the number of 
customers and the environment of the system in the steady-state regime. We aim to determine  
the steady-state probabilities for the $M/M/1$ queue under the two environments, defined as 
\begin{equation}
 q_{n,i}=\mathbb P(N=n, \mathscr{E}=i)
 =\lim_{t\to +\infty}p_{n,i}(t),\qquad n\in \mathbb N_0, \quad i=1,2.
\label{steady_state}
\end{equation}
From (\ref{equat_env1}) and (\ref{equat_env2}) one has the following difference equations:
\begin{eqnarray*}
&&-(\lambda_1+\eta_1)\,q_{0,1}+\eta_2\,q_{0,2}+\mu_1\,q_{1,1}=0, \\
&&-(\lambda_1+\mu_1+\eta_1)\,q_{n,1}+\eta_2\,q_{n,2}+\mu_1\,q_{n+1,1}+\lambda_1\,q_{n-1,1}=0,
\qquad n\in \mathbb N, \\
&&-(\lambda_2+\eta_2)\,q_{0,2}+\eta_1\,q_{0,1}+\mu_2\,q_{1,2}=0, \\
&&-(\lambda_2+\mu_2+\eta_2)\,q_{n,2}+\eta_1\,q_{n,1}+\mu_2\,q_{n+1,2}+\lambda_2\,q_{n-1,2}=0,  
\qquad n\in \mathbb N.
\end{eqnarray*}
Hence, denoting by 
\begin{equation*}
G_i(z)=\mathbb{E}[z^N  \mathbbm{1}_{\mathscr{E}=i}]=\sum_{n=0}^{+\infty}z^nq_{n,i},\qquad 0<z<1,\qquad i=1,2
\end{equation*}
the probability generating functions for the two environments in steady-state regime, one has:
\begin{eqnarray}
G_1(z)={\eta_2\,\mu_2\,z\,q_{0,2}-\mu_1\,q_{0,1}\big[\lambda_2\,z^2-(\lambda_2+\mu_2+\eta_2)\,z+\mu_2\big]\over P(z)},
\nonumber\\
\label{generating_function_1}\\
G_2(z)={\eta_1\,\mu_1\,z\,q_{0,1}-\mu_2\,q_{0,2}\big[\lambda_1\,z^2-(\lambda_1+\mu_1+\eta_1)\,z+\mu_1\big]\over P(z)},
\nonumber
\end{eqnarray}
where $P(z)$ is  the following third-degree polynomial in $z$ 
(see Eq.\ (22) of \cite{YeNa1971}): 
\begin{eqnarray}
&&\hspace*{-1.2cm}P(z)
=\lambda_1\lambda_2z^3-\big[\lambda_1\lambda_2+\lambda_1\mu_2+\lambda_1\eta_2+\mu_1\lambda_2+\eta_1\lambda_2\big]z^2
\nonumber\\
&&\hspace*{0.1cm}+\big[\lambda_1\mu_2+\mu_1\lambda_2+\mu_1\mu_2+\mu_1\eta_2+\eta_1\mu_2\big]z-\mu_1\mu_2, 
\qquad 0<z<1.
\label{third_degree_polynomial}
\end{eqnarray}
By taking into account the normalization condition $G_1(1)+G_2(1)=1$, 
from (\ref{generating_function_1}) we get:
\begin{equation}
\mu_1\,q_{0,1}+\mu_2\,q_{0,2}={ \eta_1(\mu_2-\lambda_2)+ \eta_2 (\mu_1-\lambda_1)\over \eta_1+\eta_2}\cdot
\label{equilibrium_condition}
\end{equation}
It is worth noting that Eq.~(\ref{equilibrium_condition}) is a suitable extension of the classical condition 
for the $M/M/1$ queue in the steady-state, i.e.\  $\mu\,q_{0}= \mu-\lambda$. Moreover, recalling that 
$\eta_1+\eta_2>0$, Eq.~(\ref{equilibrium_condition}) shows that the existence of the equilibrium 
distribution is guaranteed if and only if one of the following cases holds:
\begin{description}
\item{\em (i)}  $\eta_2=0$ and $\lambda_2/\mu_2<1$, 
\item{\em (ii)}  $\eta_1=0$ and $\lambda_1/\mu_1<1$, 
\item{\em (iii)}   $\eta_1>0$, $\eta_2>0$ and  $\eta_1(\mu_2-\lambda_2)+ \eta_2 (\mu_1-\lambda_1)>0$. 
\end{description}
\par
Hereafter, we  consider separately the  three cases.
\subsection*{$\bullet$ Case {\it (i)}}
If $\eta_2=0$ and $\lambda_2/\mu_2<1$ one can easily prove that 
\begin{equation}
q_{n,1}=0, \qquad 
q_{n,2}=\Bigl(1-{\lambda_2\over\mu_2}\Bigr)\,\Bigl({\lambda_2\over\mu_2}\Bigr)^n, 
\qquad n\in \mathbb N_0. 
\label{eq:eqdistribution}
\end{equation}
Therefore, a steady-state regime does not hold for the $M/M/1$ queue under the environment 
$\mathscr{E}=1$, whereas a geometric-distributed steady-state regime exists for $\mathscr{E}=2$. 
In conclusion, if $\eta_2=0$ and $\lambda_2/\mu_2<1$ then $N$ admits of a geometric 
steady-state distribution $q_n=q_{n,1}+q_{n,2}$  with parameter $\lambda_2/\mu_2$.
\subsection*{$\bullet$  Case {\it (ii)}}
If $\eta_1=0$ and $\lambda_1/\mu_1<1$, similarly to case {\em (i)}, one has 
$$
 q_{n,1}=\Bigl(1-{\lambda_1\over\mu_1}\Bigr)\,\Bigl({\lambda_1\over\mu_1}\Bigr)^n,
 \qquad q_{n,2}=0, \qquad n\in \mathbb N_0.
$$
Hence, in this case $N$ has a geometric  steady-state distribution 
$q_n=q_{n,1}+q_{n,2}$  with parameter $\lambda_1/\mu_1$.
\subsection*{$\bullet$  Case {\it (iii)}}
Let $\eta_1>0$,  $\eta_2>0$ and   $\eta_1(\mu_2-\lambda_2)+ \eta_2 (\mu_1-\lambda_1)>0$. 
These assumptions are in agreement with the conditions given in \cite{YeNa1971}.  
Denoting by $\xi_1$, $\xi_2$, $\xi_3$ the roots of $P(z)$, given in (\ref{third_degree_polynomial}), one has
\begin{eqnarray}
&&\hspace*{-0.5cm}\xi_1+\xi_2+\xi_3={\lambda_1\lambda_2+\lambda_1\mu_2+\lambda_1\eta_2+\lambda_2\mu_1+\lambda_2\eta_1\over\lambda_1\lambda_2},
\nonumber\\
&&\hspace*{-0.5cm}\xi_1\xi_2+\xi_1\xi_3+\xi_2\xi_3
={\lambda_1\mu_2+\lambda_2\mu_1+\mu_1\mu_2+\mu_1\eta_2+\mu_2\eta_1\over \lambda_1\lambda_2},
\label{prop1_roots}\\
&&\hspace*{-0.5cm}\xi_1\xi_2\xi_3={\mu_1\mu_2\over\lambda_1\lambda_2},\nonumber
\end{eqnarray}
so that $\xi_1+\xi_2+\xi_3>0$, $\xi_1\xi_2+\xi_1\xi_3+\xi_2\xi_3>0$ and $\xi_1\xi_2\xi_3>0$.  
Moreover, we note that 
\begin{equation}
(\xi_1-1)(\xi_2-1)(1-\xi_3)={ \eta_1(\mu_2-\lambda_2)+ \eta_2 (\mu_1-\lambda_1)\over \lambda_1\lambda_2}>0,
\label{prop2_roots}
\end{equation}
and that, due to (\ref{third_degree_polynomial}), 
\begin{eqnarray}
&&\hspace*{-0.5cm} P(0)=-\mu_1\mu_2<0, \qquad 
 P(1)=\eta_1(\mu_2-\lambda_2)+\eta_2(\mu_1-\lambda_1)>0, \nonumber\\
 \label{prop3_roots}\\
&&\hspace*{-0.5cm} P\Big(\frac{\mu_1}{\lambda_1}\Big)=\frac{\eta_1\mu_1\lambda_2}{\lambda_1}
 \Big(\frac{\mu_2}{\lambda_2}-\frac{\mu_1}{\lambda_1}\Big),
 \qquad 
 P\Big(\frac{\mu_2}{\lambda_2}\Big)=\frac{\eta_2\mu_2\lambda_1}{\lambda_2}
 \Big(\frac{\mu_1}{\lambda_1}-\frac{\mu_2}{\lambda_2}\Big).\nonumber
\end{eqnarray}
Hence, $P(z)$ has three positive roots, two of them greater than 1 and one less than 1. 
Hereafter we show that the present method allows us to express the distribution 
of interest in closed form, as a generalized mixture of geometric distributions. Hence, 
we assume that $\xi_1>1$, $\xi_2>1$ and $0<\xi_3<1$, and thus  
$P(z)=\lambda_1\lambda_2(z-\xi_1)(z-\xi_2)(z-\xi_3)$. 
%
%
\begin{proposition}\label{prop:ssprob}
If $\eta_1>0$, $\eta_2>0$ and $\eta_1(\mu_2-\lambda_2)+ \eta_2 (\mu_1-\lambda_1)>0$,  
then the joint steady-state probabilities of ${\bf N}=(N,\mathscr{E})$ can be expressed in terms of the 
roots $\xi_1>1$, $\xi_2>1$ and $0<\xi_3<1$ of the polynomial (\ref{third_degree_polynomial}) as follows: 
\begin{equation}
q_{n,i}={\eta_{3-i}\over\eta_1+\eta_2} \Bigl[ A_i\,\mathbb{P}(V_1=n)+(1-A_i)\,\mathbb{P}(V_2=n)\Bigr],
\qquad n\in\mathbb{N}_0,i=1,2,
\label{mixture_discrete_environments}
\end{equation}
where  $P(V_i=n)=(1-1/\xi_i)(1/\xi_i)^n$ $\;(n\in\mathbb{N}_0,\;i=1,2)$ and 
\begin{equation}
A_i={\xi_1\xi_3\bigl[\eta_1(\mu_2-\lambda_2)+\eta_2(\mu_1-\lambda_1)\bigr]
\over\lambda_i\mu_i(1-\xi_3)(\xi_1-1)(\xi_1-\xi_2)}\,{\mu_i-\lambda_i\xi_2\over \mu_{3-i}-\lambda_{3-i}\xi_3},
\quad i=1,2.
\label{coef_mixture_discrete}
\end{equation}
\end{proposition}
\begin{proof} Since $P(\xi_3)=0$, to ensure the  convergence of the probability generating functions (\ref{generating_function_1}), 
we impose that their numerators   tend to zero as $z\to\xi_3$. Hence, by virtue of (\ref{equilibrium_condition}), 
one has (see also Eqs.\ (26) and (27) of \cite{YeNa1971}):  
\begin{eqnarray}
&&q_{0,1}={\eta_2\xi_3\over \mu_1(1-\xi_3)(\mu_2-\lambda_2\xi_3)}\;{ \eta_1(\mu_2-\lambda_2)+ \eta_2 (\mu_1-\lambda_1)\over \eta_1+\eta_2},\nonumber\\
&&\label{equil_probabilities_zero}\\
&&q_{0,2}={\eta_1\xi_3\over \mu_2(1-\xi_3)(\mu_1-\lambda_1\xi_3)}\;{ \eta_1(\mu_2-\lambda_2)+ \eta_2 (\mu_1-\lambda_1)\over \eta_1+\eta_2}\cdot\nonumber
\end{eqnarray}
Note that, due to (\ref{prop1_roots}), (\ref{prop2_roots}) and (\ref{prop3_roots}), 
one has $\xi_3\neq \mu_1/\lambda_1$ and $\xi_3\neq \mu_2/\lambda_2$. 
Making use of (\ref{equil_probabilities_zero}), from (\ref{generating_function_1}) one finally obtains:
\begin{eqnarray}
G_1(z)={\eta_2 \bigl[\eta_1(\mu_2-\lambda_2)+ \eta_2 (\mu_1-\lambda_1)\bigr]
\over (1-\xi_3)(\mu_2-\lambda_2\xi_3)(\eta_1+\eta_2)}\;
{\mu_2-\lambda_2\xi_3\,z\over \lambda_1\lambda_2(z-\xi_1)(z-\xi_2)},\nonumber\\
\label{generating_function_2}\\
G_2(z)={\eta_1 \bigl[\eta_1(\mu_2-\lambda_2)+ \eta_2 (\mu_1-\lambda_1)\bigr]
\over (1-\xi_3)(\mu_1-\lambda_1\xi_3)(\eta_1+\eta_2)}\;
{\mu_1-\lambda_1\xi_3\,z\over \lambda_1\lambda_2(z-\xi_1)(z-\xi_2)}.\nonumber
\end{eqnarray}
Expanding $G_1(z)$ and $G_2(z)$, given in  (\ref{generating_function_2}), 
in power series of $z$, one finally is led to  
\begin{eqnarray}
&&\hspace*{-0.5cm}q_{n,1}={\eta_2\xi_3\over \mu_1(1-\xi_3)(\mu_2-\lambda_2\xi_3)}\;{ \eta_1(\mu_2-\lambda_2)+ \eta_2 (\mu_1-\lambda_1)\over \eta_1+\eta_2}\nonumber\\
&&\hspace*{0.5cm}\times {1\over(\xi_1\xi_2)^n}\Bigl\{ {\xi_1^{n+1}-\xi_2^{n+1}\over \xi_1-\xi_2}
-{\mu_1\over\lambda_1}\;{\xi_1^n-\xi_2^n\over \xi_1-\xi_2}\Bigr\},
\qquad n\in\mathbb N_0,\nonumber\\
&&\label{equil_probabilities}\\
&&\hspace*{-0.5cm}q_{n,2}={\eta_1\xi_3\over \mu_2(1-\xi_3)(\mu_1-\lambda_1\xi_3)}\;{ \eta_1(\mu_2-\lambda_2)+ \eta_2 (\mu_1-\lambda_1)\over \eta_1+\eta_2}\nonumber\\
&&\hspace*{0.5cm}\times {1\over(\xi_1\xi_2)^n}\Bigl\{ {\xi_1^{n+1}-\xi_2^{n+1}\over \xi_1-\xi_2}
-{\mu_2\over\lambda_2}\;{\xi_1^n-\xi_2^n\over \xi_1-\xi_2}\Bigr\},
\qquad n\in\mathbb N_0.\nonumber
\end{eqnarray}
From (\ref{equil_probabilities}) one obtains immediately (\ref{mixture_discrete_environments}).
\hfill\fine
\end{proof}
We note that if $\eta_1(\mu_2-\lambda_2)+ \eta_2 (\mu_1-\lambda_1)>0$, the steady-state probabilities 
do not depend on the initial conditions  (\ref{initial_condition}), i.e. on the probability $p$.
\par
By virtue of (\ref{prop2_roots}), from  (\ref{generating_function_2}) one obtains 
(cf.\ also Eqs.\ (17) of \cite{YeNa1971}): 
\begin{equation}
 \mathbb{P}(\mathscr{E}=i)\equiv G_i(1)
 =\sum_{n=0}^{+\infty}q_{n,i}={\eta_{3-i}\over \eta_1+\eta_2},
\qquad 
i=1,2.
\label{eq:probE}
\end{equation}
From Proposition~\ref{prop:ssprob}, we have that  $\mathbb{P}(N=n|\mathscr{E}=1)$ and 
$\mathbb{P}(N=n|\mathscr{E}=2)$ $(n\in\mathbb{N}_0)$
are  both generalized mixtures of two geometric probability distributions of parameters 
$1/\xi_1$ and $1/\xi_2$, respectively (see Navarro \cite{Navarro2016} for details on generalized mixtures). 
\par
Making use of Proposition~\ref{prop:ssprob} and of (\ref{eq:probE}),  
we determine the conditional means in a straightforward manner:
\begin{equation}
 \mathbb{E}[N|\mathscr{E}=i]=\sum_{n=1}^{+\infty}n {q_{n,i}\over\mathbb{P}(\mathscr{E}=i)}=
 {A_i\over \xi_1-1}+{1-A_i\over \xi_2-1},\qquad i=1,2.
 \label{expectations_discrete_environments}
\end{equation}
%
%
\begin{corollary}\label{corollary1}
Under the assumptions of Proposition~\ref{prop:ssprob}, for $n\in\mathbb{N}_0$ one obtains 
the steady-state probabilities of $N$: 
\begin{equation}
q_n=q_{n,1}+q_{n,2}
={\eta_2A_1+\eta_1A_2\over\eta_1+\eta_2}\mathbb{P}(V_1=n)+\Bigl[1-{\eta_2A_1+\eta_1A_2\over\eta_1+\eta_2}\Bigr]\mathbb{P}(V_2=n),
\label{mixture_discrete_system}
\end{equation}
where $A_1$ and $A_2$ are provided in (\ref{coef_mixture_discrete}). 
\end{corollary}
Eq.~(\ref{mixture_discrete_system}) shows that also $q_n$ is a generalized mixture of two geometric probability distributions of 
parameters $1/\xi_1$ and $1/\xi_2$, respectively, 
so that
\begin{equation}
\mathbb{E}(N)={\eta_2A_1+\eta_1A_2\over\eta_1+\eta_2}{1\over \xi_1-1}+\Bigl[1-{\eta_2A_1+\eta_1A_2\over\eta_1+\eta_2}\Bigr]{1\over\xi_2-1}\cdot
 \label{expectations_discrete_system}
\end{equation}
This result is in agreement with Eq.\ (33) of \cite{YeNa1971}. 
\par
Figure~\ref{fig2} shows the steady-state probabilities $q_{n,1},q_{n,2}$ (on the left) and 
$q_n=q_{n,1}+q_{n,2}$ (on the right), obtained via Proposition \ref{prop:ssprob} and Corollary~\ref{corollary1}, for 
$\lambda_1=1$, $\mu_1=0.5$, $\lambda_2=1$, $\mu_2=2$, $\eta_1 = 0.1$ and $\eta_2 = 0.08$. 
The roots of polynomial  (\ref{third_degree_polynomial}) can be evaluated by means of 
MATHEMATICA$^{\footnotesize{\rm \textregistered}}$, so that  
$\xi_1=2.16716$, $\xi_2=1.08919$, $\xi_3=0.423647$.
%
\begin{figure}[t]  
\centering
\includegraphics[scale=0.6]{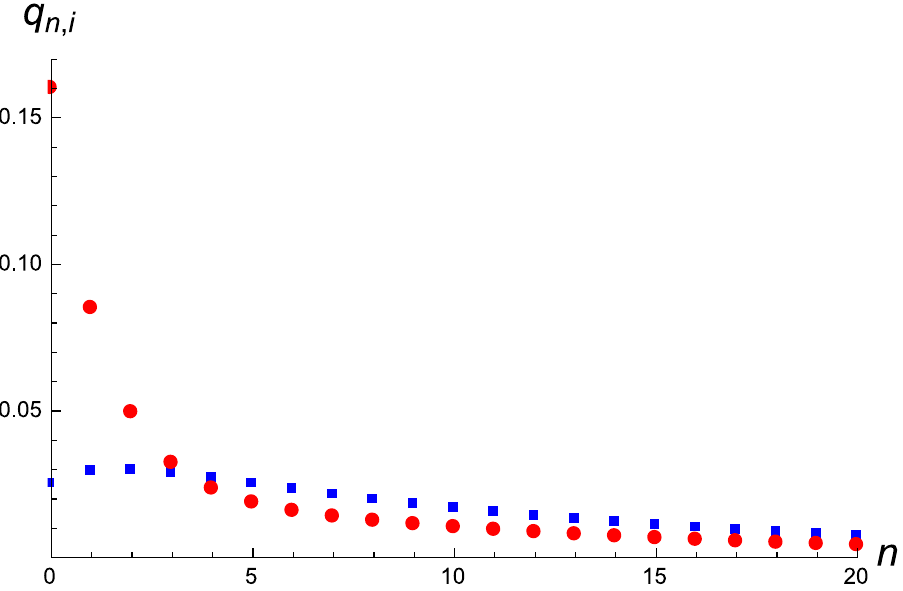}
\hspace*{4mm}
\includegraphics[scale=0.6]{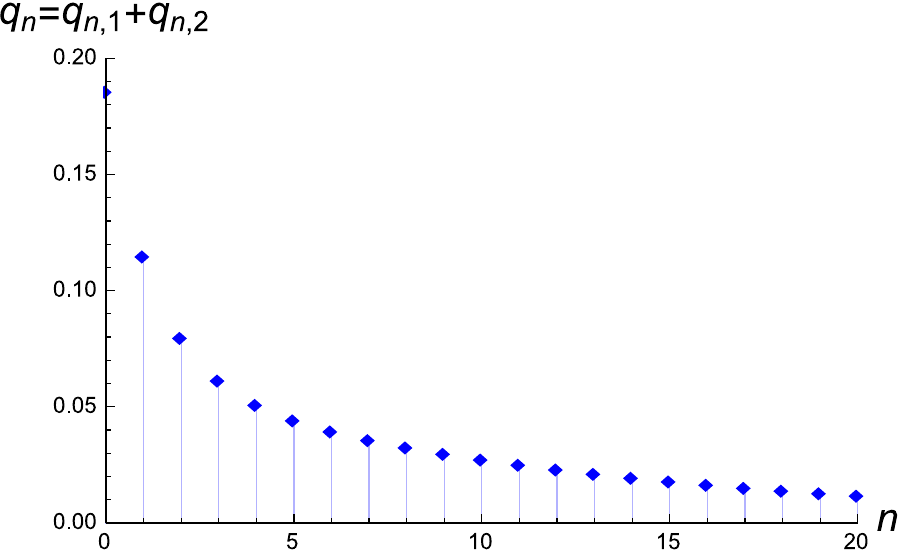}\\
\caption{Plots of probabilities $q_{n,1}$ (square) and $q_{n,2}$ (circle), on the left, and 
$q_n=q_{n,1}+q_{n,2}$, on the right, for  $\lambda_1=1$, $\mu_1=0.5$, $\lambda_2=1$, 
$\mu_2=2$, $\eta_1 = 0.1$ and $\eta_2 = 0.08$.}
\label{fig2}
\end{figure}
%
\begin{figure}[t]  
\centering
\includegraphics[scale=0.6]{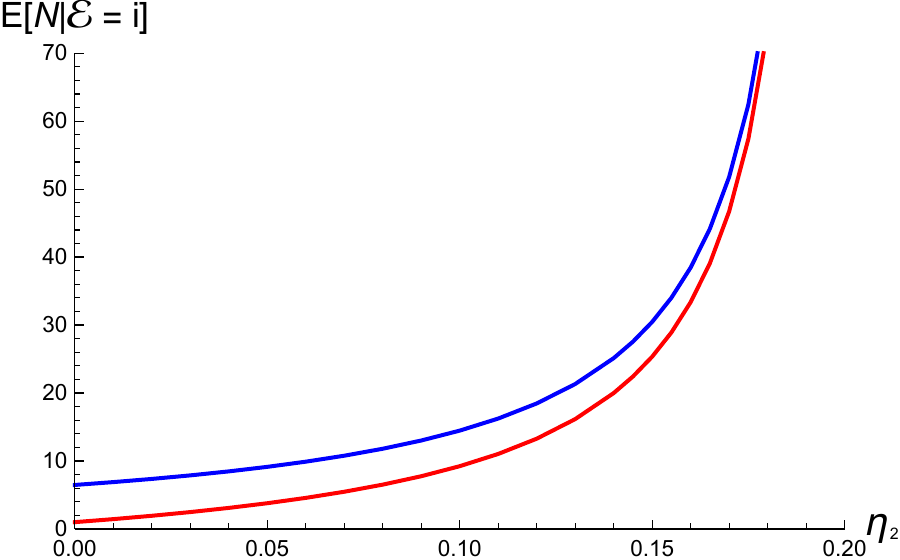}
\hspace*{4mm}
\includegraphics[scale=0.6]{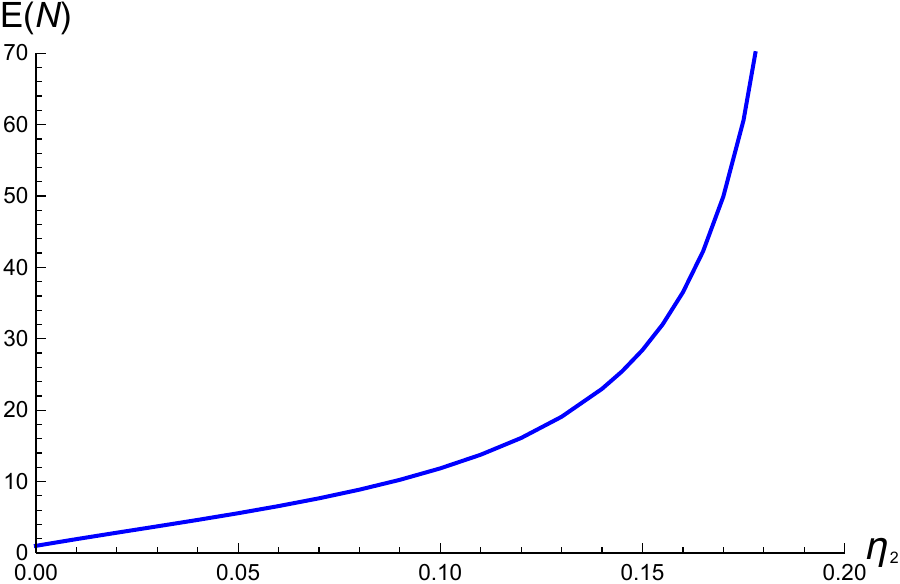}\\
\caption{For $\lambda_1=1$, $\mu_1=0.5$, $\lambda_2=1$, $\mu_2=2$, $\eta_1 = 0.1$ and $0\leq\eta_2 <0.2$
the  conditional means $\mathbb{E}[N|\mathscr{E}=i]$, given in  (\ref{expectations_discrete_environments}), are plotted on the left
for $i=1$ (top) and $i=2$ (bottom), whereas the mean $\mathbb{E}(N)$ is plotted on the right.
}
\label{Figure3}
\end{figure} 
\par
Figure~\ref{Figure3} gives, on the left, a plot of the conditional means, obtained in (\ref{expectations_discrete_environments}),  
for a suitable choice of the parameters, showing 
that $\mathbb{E}[N|\mathscr{E}=i]$ is increasing in $\eta_2$, for $i=1,2$. 
The  mean $\mathbb{E}(N)$, obtained via (\ref{expectations_discrete_system}),  
is plotted as function of $\eta_2$ on the right of Figure~\ref{Figure3}.
%
\begin{figure}[t]  
\centering
\includegraphics[scale=0.6]{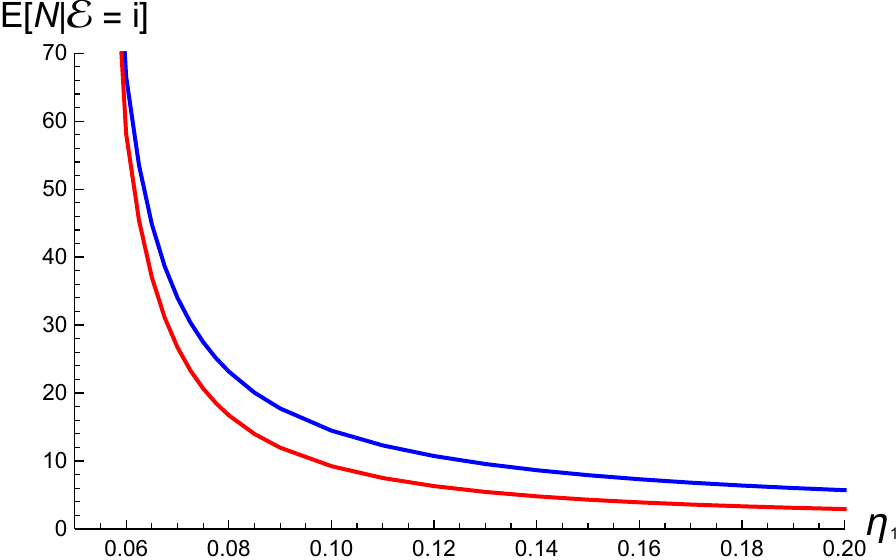}
\hspace*{4mm}
\includegraphics[scale=0.6]{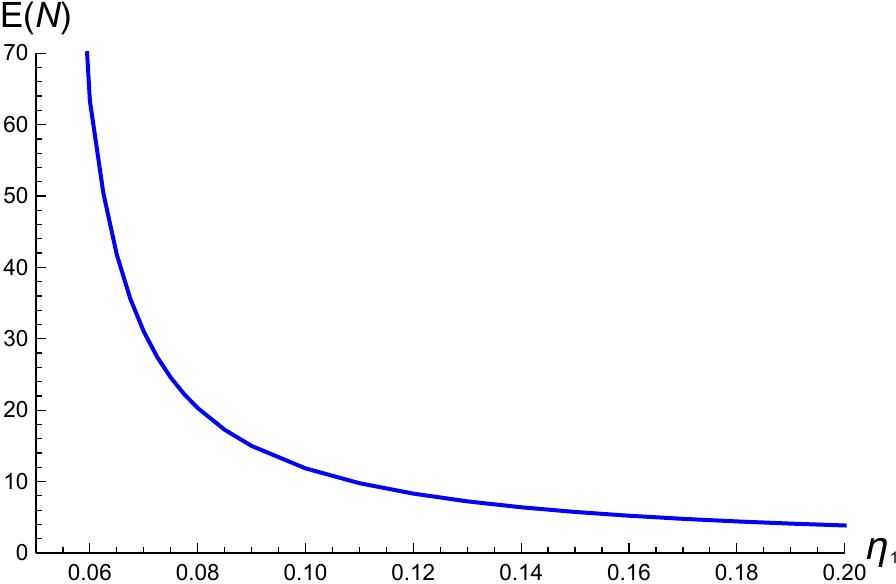}\\
\caption{For $\lambda_1=1$, $\mu_1=0.5$, $\lambda_2=1$, $\mu_2=2$, $\eta_2 = 0.1$ and $\eta_1 >0.05$
the  conditional means $\mathbb{E}[N|\mathscr{E}=i]$, given in  (\ref{expectations_discrete_environments}), are plotted on the left
for $i=1$ (top) and $i=2$ (bottom), whereas the mean $\mathbb{E}(N)$ is plotted on the right.
}
\label{Figure4}
\end{figure} 
\par
Similarly, on the left of Figure~\ref{Figure4} are plotted  the conditional means  for a suitable choice of the parameters, 
showing that $\mathbb{E}[N|\mathscr{E}=i]$ is decreasing in $\eta_1$, for $i=1,2$. 
The  mean $\mathbb{E}(N)$  is plotted as function of $\eta_1$ on the right of Figure~\ref{Figure4}.
\par
We remark that in the cases considered in Figures~\ref{Figure3} and \ref{Figure4}, the parameters satisfy the 
condition $\eta_1(\mu_2-\lambda_2)+ \eta_2 (\mu_1-\lambda_1)>0$, and thus the joint 
steady-state probabilities of ${\bf N}$ there exist due to Proposition \ref{prop:ssprob}. 
Moreover, in the cases considered in Figures~\ref{Figure3} and \ref{Figure4}, one has $\lambda_1>\mu_1$ and $\lambda_2<\mu_2$, so that in 
absence of switching  the queue should not admit a steady-state regime in the first environment, 
whereas a steady-state regime should hold in the second environment. This example shows that the switching 
mechanism is useful to obtain stationarity, in the sense that the alternating $M/M/1$ 
queue may admit a steady-state regime even if one of the alternating environments does not. 
\par
Making use of Proposition~\ref{prop:ssprob} and Eq.\ (\ref{eq:probE}) it is easy to determine 
the (Shannon)  entropies 
\begin{eqnarray}
&& H[N|\mathscr{E}=i]=-\sum_{n=1}^{+\infty} \frac{q_{n,i}}{\mathbb{P}(\mathscr{E}=i)}
 \log  \Big[ \frac{q_{n,i}}{\mathbb{P}(\mathscr{E}=i)}\Big],\qquad i=1,2, 
 \label{eq:condentr}\\
 && H(N)=-\sum_{n=1}^{+\infty}q_{n} \log  q_n, \label{eq:entr}
\end{eqnarray}
where `$\log$' means natural logarithm. 
\begin{figure}[t]  
\centering
\includegraphics[scale=0.6]{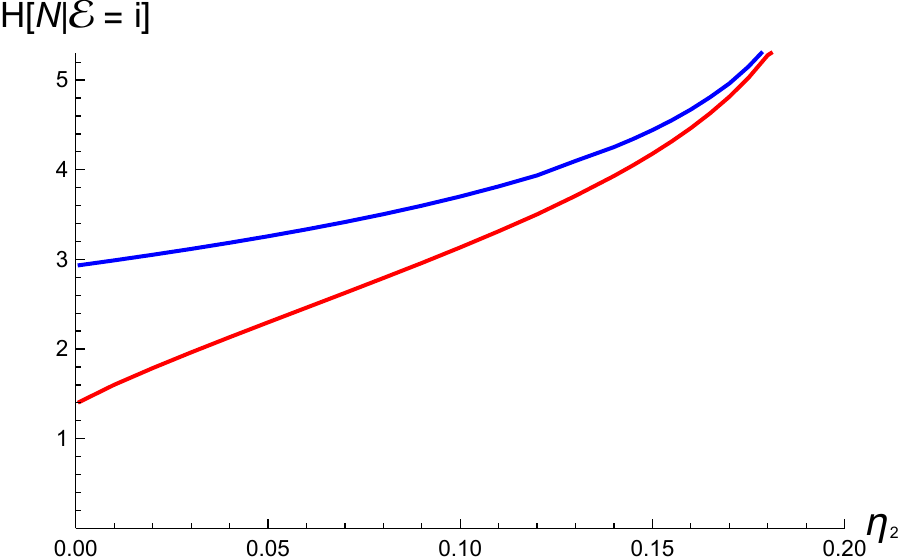}
\hspace*{4mm}
\includegraphics[scale=0.6]{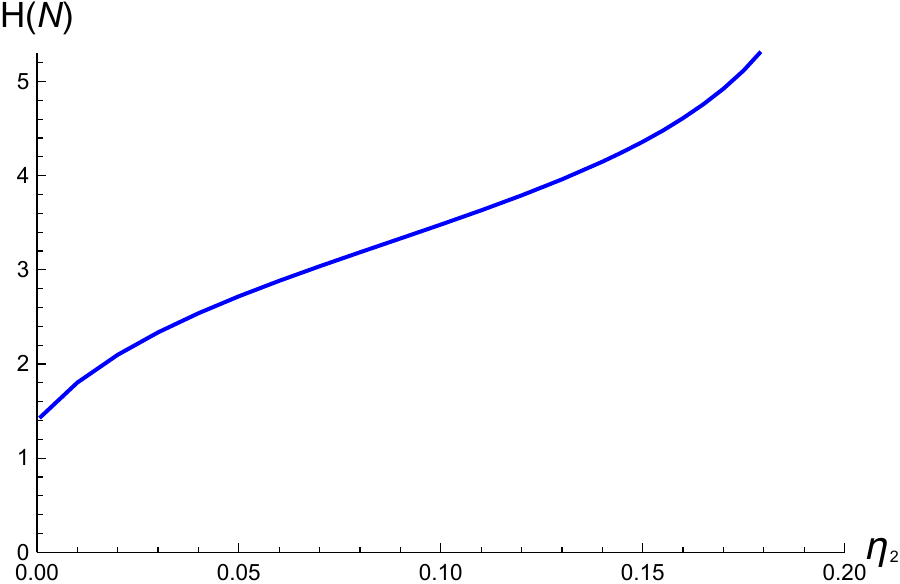}\\
\caption{On the left, the conditional entropy $H[N|\mathscr{E}=i]$, given in (\ref{eq:condentr}), is plotted 
with $\lambda_1=1$, $\mu_1=0.5$, $\lambda_2=1$, $\mu_2=2$, 
$\eta_1 = 0.1$ and $0\leq \eta_2 < 0.2$, for $i=1$ (top) and $i=2$ (bottom). 
On the right, the entropy $H(N)$, given in (\ref{eq:entr}),  is plotted for the same choices of parameters.}
\label{Figure5}
\end{figure} 
The conditional entropy (\ref{eq:condentr})  is a measure of interest in queueing, since it gives the average amount of information 
that is gained when the steady-state number of customers $N$ in the queue is observed 
given that the system is in environment $\mathscr{E}=i$. 
In Figure~\ref{Figure5} we plot the conditional entropy (\ref{eq:condentr}) (on the left)  and the entropy (\ref{eq:entr}) (on the right)  
 for the same case of Figure~\ref{Figure3}, showing that $H[N|\mathscr{E}=i]$ $(i=1,2)$ and $H[N)$  are increasing in $\eta_2$. 
 Furthermore,   for the same choices of Figure~\ref{Figure4}, in Figure~\ref{Figure6} the entropies $H[N|\mathscr{E}=i]$ $(i=1,2)$ and $H[N)$  
 are plotted, showing that are decreasing in $\eta_1$. 
%
\begin{figure}[t]  
\centering
\includegraphics[scale=0.6]{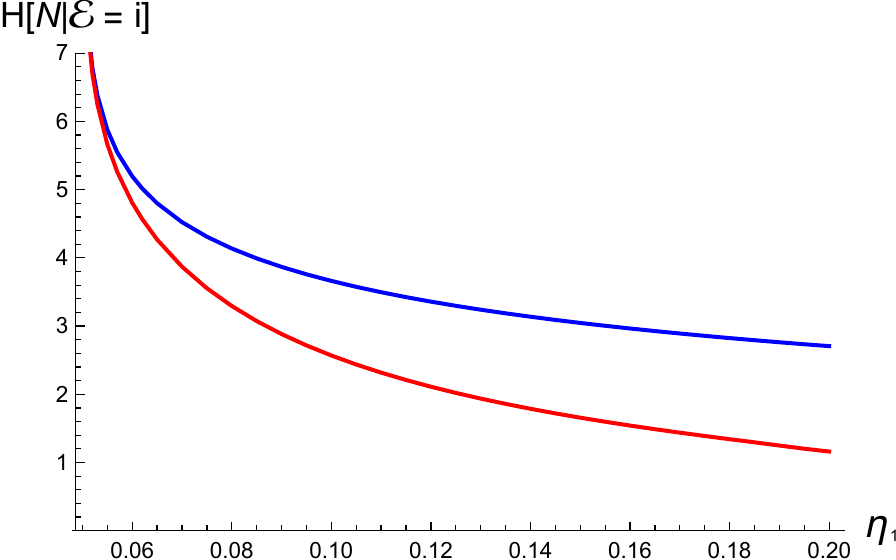}
\hspace*{4mm}
\includegraphics[scale=0.6]{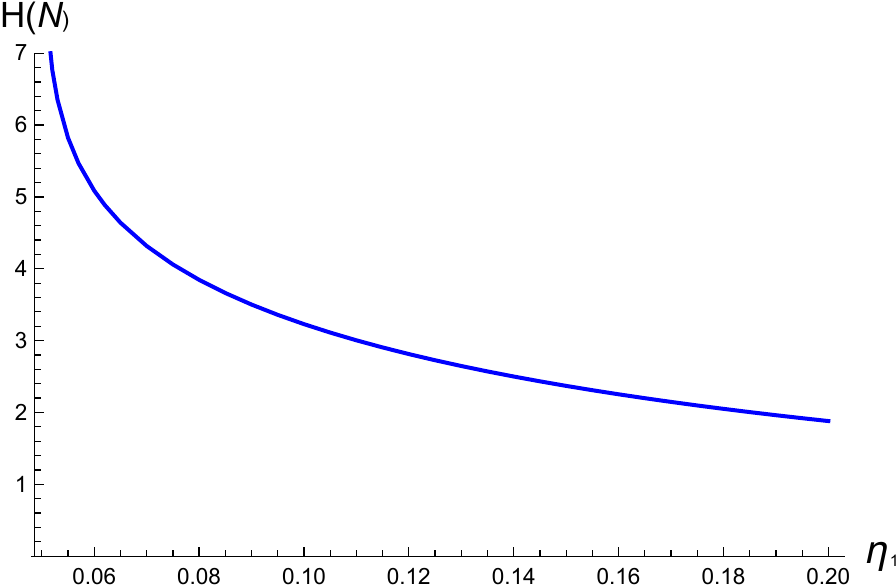}\\
\caption{On the left, $H[N|\mathscr{E}=i]$ is plotted 
with $\lambda_1=1$, $\mu_1=0.5$, $\lambda_2=1$, $\mu_2=2$, 
$\eta_2 = 0.1$ and $ \eta_1 > 0.05$, for $i=1$ (top) and $i=2$ (bottom). 
On the right, $H(N)$ is plotted for the same choices of parameters.}
\label{Figure6}
\end{figure} 
Moreover, recalling (\ref{mixture_discrete_environments}), (\ref{eq:probE}) and (\ref{mixture_discrete_system})   one can also define
the  following entropies 
\begin{eqnarray}
&&H[\mathscr{E}|N=n]=-\sum_{i=1}^2{q_{n,i}\over q_n}\log {q_{n,i}\over q_n},\qquad n\in\mathbb{N}_0,
\label{con_entr_n}\\
&&H(\mathscr{E})=-\sum_{i=1}^2{\eta_{3-i}\over \eta_1+\eta_2}\log {\eta_{3-i}\over \eta_1+\eta_2}
\label{entr_envir}
\end{eqnarray}
Note that $0\leq H[\mathscr{E}|N=n]\leq \log 2$ and $0\leq H(\mathscr{E})\leq \log 2$.
The conditional entropy (\ref{con_entr_n}) gives  the average amount of information on the environment when the  number of customers 
$N=n$ in the queue is observed, whereas the entropy (\ref{entr_envir}) gives  the average amount of information on the environment.
Note that, under the assumptions of Proposition~\ref{prop:ssprob}, from (\ref{con_entr_n}) one has
\begin{eqnarray}
&&\hspace*{-1.5cm}H_{\infty}=\lim_{n\to +\infty}H[\mathscr{E}|N=n]\nonumber\\
&&\hspace*{-0.8cm}=-\sum_{i=1}^2 {\eta_{3-i}\,(1-A_i)\over \eta_1\,(1-A_2)+\eta_2\,(1-A_1)}
\log {\eta_{3-i}\,(1-A_i)\over \eta_1\,(1-A_2)+\eta_2\,(1-A_1)},
\label{limit_entropy_cond}
\end{eqnarray}
with $A_i$ defined in (\ref{coef_mixture_discrete}). Therefore, the conditional average amount of information on the environment 
tends to  the value given in (\ref{limit_entropy_cond}) when the  number of customers increases. 
%
\begin{figure}[t]  
\centering
\subfigure[$\eta_1=0.1$]{ \includegraphics[width=0.42\textwidth,keepaspectratio]{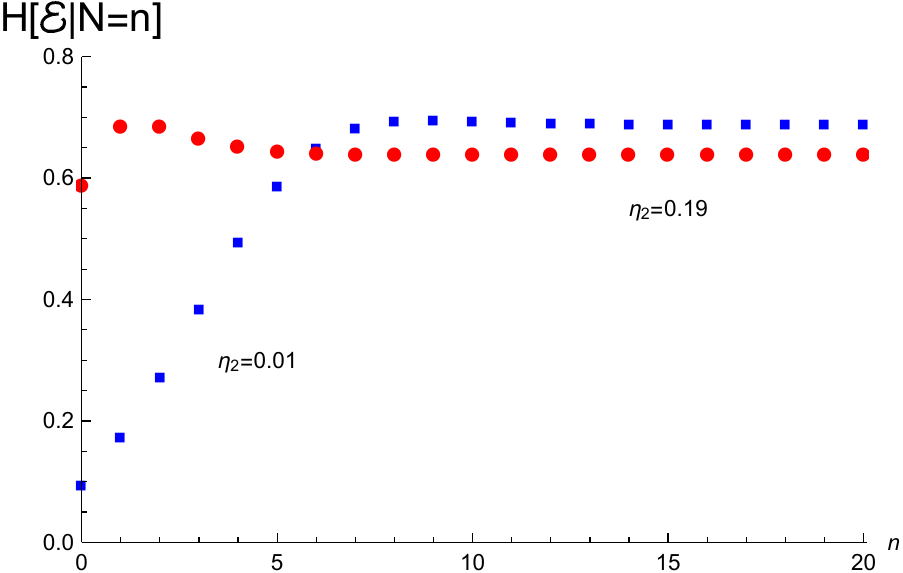}}
\hspace*{4mm}
\subfigure[$\eta_2=0.1$]{ \includegraphics[width=0.42\textwidth,keepaspectratio]{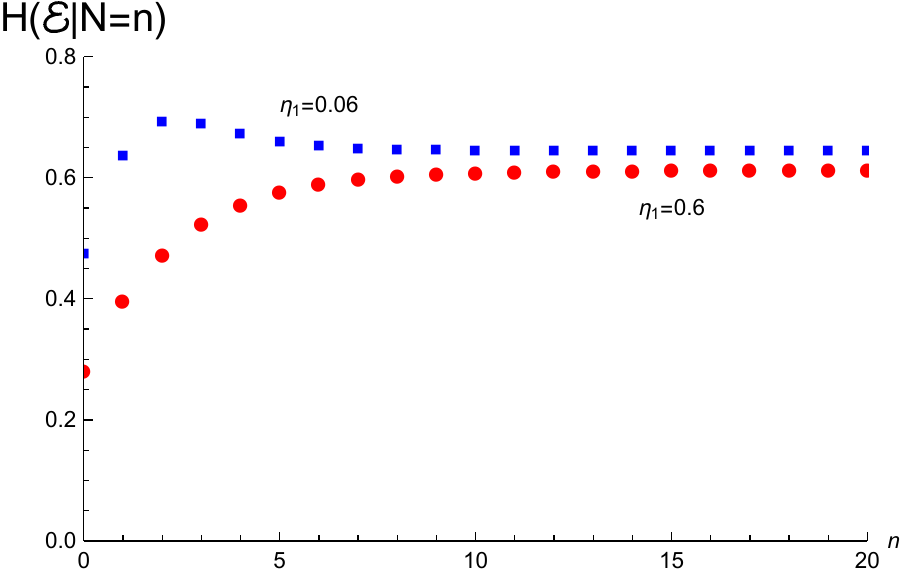}}\\
\caption{Plots of  $H[\mathscr{E}|N=n]$ for $\lambda_1=1.0,\mu_1=0.5,\lambda_2=1,\mu_2=2.0$ and $n=1,2,\ldots,20$, with different choices of $\eta_1,\eta_2$.}
\label{fig7}
\end{figure}
In Figure~\ref{fig7} the conditional entropies (\ref{con_entr_n}) are shown for some choices of parameters. In particular, in  Figure~\ref{fig7}(a) 
$H[\mathscr{E}|N=0]=0.0923799$, $H_{\infty}=0.686201$ and $H(\mathscr{E})=0.304636$ for $\eta_2=0.01$, whereas  when  $\eta_2=0.19$ one has $H[\mathscr{E}|N=0]=0.5898$,
$H_{\infty}=0.639399$ and $H(\mathscr{E})=0.644186$. Furthermore, in  Figure~\ref{fig7}(b) one has $H[\mathscr{E}|N=0]=0.473177$, $H_{\infty}=0.643289$ and $H(\mathscr{E})=0.661563$ 
for $\eta_1=0.06$, whereas  when  $\eta_1=0.6$ it results $H[\mathscr{E}|N=0]=0.281199$, $H_{\infty}=0.613038$ and $H(\mathscr{E})=0.410116$.   
From Figure~\ref{fig7}, we note that if $\eta_1>\eta_2$ then $H[\mathscr{E}|N=0]<H(\mathscr{E})<H_{\infty}$,  whereas   when $\eta_1<\eta_2$ it follows 
$H[\mathscr{E}|N=0]<H_{\infty}<H(\mathscr{E})$.
%
\section{Analysis of case $\eta_2=0$}\label{section3}
In the general case it is hard to determine the transient distribution of ${\bf N}(t)$. 
Hence, we limit ourselves to the analysis of the  system with  $\eta_2=0$ and with the initial state 
given in (\ref{eq:initcondit_N}). Figure~\ref{fig:5} shows the state diagram of ${\bf N}(t)$ in this special  case, where
only transitions from the first to the second environment are allowed. 
%
\begin{figure}[t]  
\centering
\includegraphics[scale=0.45]{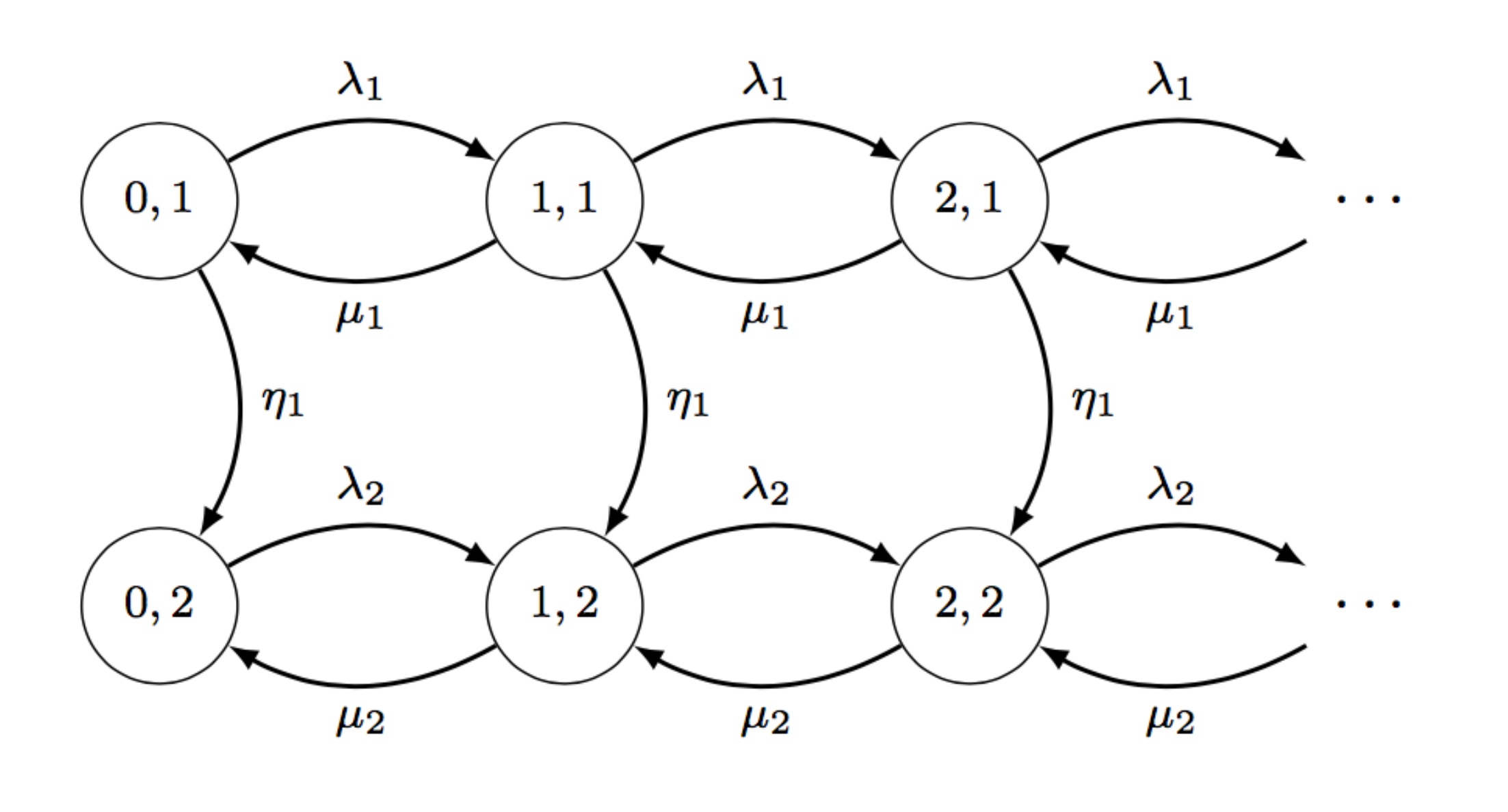}\\
\caption{The state diagram of the Markov chain ${\bf N}(t)$ when $\eta_2=0$.}
\label{fig:5}
\end{figure}
\par
We  remark that,  the case $\eta_1=0$ can be studied similarly by symmetry. 
\subsection{Transient probabilities}
%
Hereafter, we express the transient  probabilities (\ref{eq:transprobab}) in terms of 
the analogue probabilities $\widehat p_{j,n}^{(i)}(t)$ of two $M/M/1$ queueing systems 
$\widehat N^{(i)}(t)$, $t\geq 0$, $i=1,2$,  characterized by arrival rate $\lambda_i$ 
and service rate $\mu_i$. Note that, for $j,n\in\mathbb{N}_0$, $t\geq 0$ and $i=1,2$, we have 
(see, e.g.,  Zhang and Coyle \cite{ZC1991}, or Eq.~(32) of Giorno {\em et al.}\ \cite{GNS2014})
\begin{eqnarray}
\widehat p_{j,n}^{(i)}(t) & = & e^{-(\lambda_i+\mu_i)t}
\Big\{\Big(\frac{\lambda_i}{\mu_i}\Big)^{(n-j)/2} I_{n-j}(2t\sqrt{\lambda_i \mu_i})
\nonumber \\
& + & \Big(\frac{\lambda_i}{\mu_i}\Big)^{(n-j-1)/2} I_{n+j+1}(2t\sqrt{\lambda_i \mu_i})
\label{eq:trprobMM1} \\
& + &  \Big(1-\frac{\lambda_i}{\mu_i}\Big)
\Big(\frac{\lambda_i}{\mu_i}\Big)^{n}\sum_{k=n+j+2}^{\infty} 
\Big(\frac{\mu_i}{\lambda_i}\Big)^{k/2} I_{k}(2t\sqrt{\lambda_i \mu_i})\Big\},
\nonumber
\end{eqnarray}
where $I_\nu(z)$ denotes the modified Bessel function of the first kind. Moreover, making use of  
Eq.~(49), pag. 237, of Erd\'elyi {\em et al.}\ \cite{Erdelyi1}, for $n\in\mathbb{N}_0$, $t\geq 0$ and $i=1,2$, we have   
\begin{equation} 
 \widehat p_{0,n}^{(i)}(t) =  \frac{1}{\mu_i}\Big(\frac{\lambda_i}{\mu_i}\Big)^n \frac{1}{t} e^{-(\lambda_i+\mu_i)t} 
 \sum_{k=n+1}^{\infty} k \Big(\frac{\lambda_i}{\mu_i}\Big)^{k/2} I_k(2t\sqrt{\lambda_i\mu_i}).
 \label{eq:trprob0MM1}
\end{equation} 
\begin{proposition}
Let $\eta_2=0$. For all $t\geq 0$ and $j,n\in \mathbb{N}_0$, the transition probabilities of ${\bf N}(t)$ can be 
expressed as:
\begin{eqnarray}
 &&\hspace*{-1.0cm}p_{n,1} (t)  = p\, e^{-\eta_1 t} \widehat p_{j,n}^{(1)}(t),
 \label{eq:formulapn1} \\
 &&\hspace*{-1.0cm}p_{n,2} (t)  = (1-p)\, \widehat p_{j,n}^{(2)}(t)+ p\, \eta_1 \sum_{k=0}^{+\infty}
 \int_0^t e^{-\eta_1\, \tau}\, \widehat p_{j,k}^{(1)}(\tau)\,\widehat p_{k,n}^{(2)}(t-\tau) \;d\tau,
 \label{eq:formulapn2}
\end{eqnarray}
where $\widehat p_{j,n}^{(i)}(t)$ are provided in (\ref{eq:trprobMM1}) and (\ref{eq:trprob0MM1}).
\end{proposition}
\begin{proof}
It follows from   (\ref{equat_env1}) and  (\ref{equat_env2}), recalling the initial conditions (\ref{initial_condition}).
\hfill\fine
\end{proof}
\par
In the case $\eta_2=0$, at most one switch can occur (from environment 1 to environment 2). Hence, 
Eq.~(\ref{eq:formulapn1}) is also obtainable  by noting that $p_{n,1} (t)$ can be viewed as 
the  probability that process ${\bf N}(t)$ is located at $(n,1)$ at time $t$, starting from $(j,1)$ 
at time $0$, with probability $p$, and that no `switches' occurred up to time  $t$. 
Similarly, resorting to  the total probability law, 
Eq.~(\ref{eq:formulapn2}) is recovered by taking into account that 
process ${\bf N}(t)$ is located at $(n,2)$ at time $t$ in two cases: 
\begin{description}
\item{-}  when starting from $(j,2)$ at time $0$, with probability $1-p$,
and performing a transition from $(j,2)$ to $(n,2)$ at time $t$,  
\item{-} when starting from $(j,1)$ at time $0$, with probability $p$, 
then performing a transition from  $(j,1)$ to $(k,1)$ at time $\tau$ 
(with $k\in \mathbb{N}_0$ and $0<\tau<t$), then switching from environment 1 to 
environment 2 at time $\tau$, and finally performing a transition 
from state $(k,2)$ to $(n,2)$ in the time interval $(\tau,t)$. 
\end{description}
%
\subsection{First-passage time problems}
We now analyze the first-passage time through zero state of the system when $\eta_2=0$. To this aim, first we define a new two-dimensional 
stochastic process $\{ \widetilde {\bf N}(t)=[\widetilde N(t),{\mathscr{E}}(t)],t\geq 0\}$ 
whose state diagram is given in Figure~\ref{fig:chain_busy}. This process is obtained from ${\bf N}(t)$ by removing all the transitions from $(0,i)$, $i=1,2$.
In this case only transitions from the first to the second environment are allowed. Moreover,  
$(0,1)$ and $(0,2)$ are  absorbing states for the process $\widetilde {\bf N}(t)$. 
%
%
\begin{figure}[t]  
\centering
\includegraphics[scale=0.45]{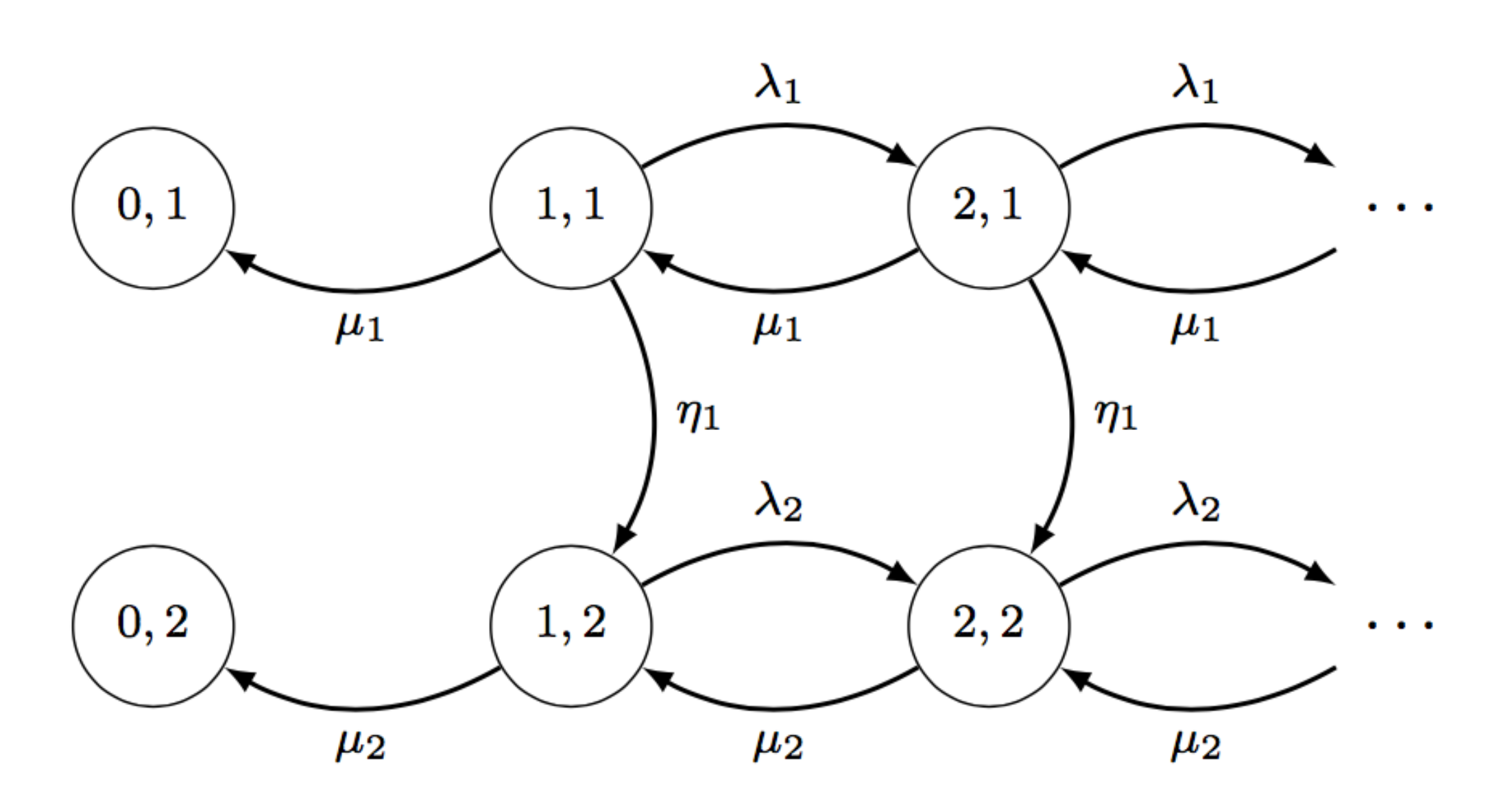}\\
\caption{The state diagram of the  birth-death process for the busy period.}
\label{fig:chain_busy}
\end{figure}
We denote by 
\begin{equation}
 \gamma_{n,i}(t)=\mathbb P[\widetilde {\bf N}(t)=(n,i)], 
 \qquad  n\in \mathbb{N}_0, 
 \quad i=1,2, \quad t\geq 0
 \label{absorb_prob_discr}
\end{equation}  
the state probabilities of the new process,  where 
\begin{equation}
\widetilde {\bf N}(0)=
\left\{\begin{array}{ll}
(j,1),&\;{\rm with\;probability}\; p,\\
(j,2),&\;{\rm with\;probability}\;1-p.
\end{array}\right.
 \label{eq:initcondit}
\end{equation}
Since $\eta_2=0$ the following equations hold:
\begin{eqnarray}
&&\hspace*{-0.5cm}{d \gamma_{0,1}(t)\over dt}= \mu_1\,\gamma_{1,1}(t),
\nonumber\\
&&\hspace*{-0.5cm} {d \gamma_{1,1}(t)\over dt}
=-(\lambda_1+\mu_1+\eta_1)\,\gamma_{1,1}(t)+\mu_1\,\gamma_{2,1}(t),
\label{equat_genv1}\\
&&\hspace*{-0.5cm} {d \gamma_{n,1}(t)\over dt}
=-(\lambda_1+\mu_1+\eta_1)\,\gamma_{n,1}(t)+\mu_1\,\gamma_{n+1,1}(t)+\lambda_1\,\gamma_{n-1,1}(t),
\nonumber\\
&& \hspace*{8cm} 
n=2,3,\ldots\nonumber
\end{eqnarray}
and
\begin{eqnarray}
&&\hspace*{-0.5cm} {d \gamma_{0,2}(t)\over dt}=\mu_2\,\gamma_{1,2}(t), \nonumber\\
&&\hspace*{-0.5cm}{d \gamma_{1,2}(t)\over dt}=-(\lambda_2+\mu_2)\,\gamma_{1,2}(t)
+\eta_1\,\gamma_{1,1}(t)+\mu_2\,\gamma_{2,2}(t), 
\label{equat_genv2}\\
&&\hspace*{-0.5cm}{d \gamma_{n,2}(t)\over dt}=-(\lambda_2+\mu_2)\,\gamma_{n,2}(t)+\eta_1\,\gamma_{n,1}(t)
+\mu_2\,\gamma_{n+1,2}(t)+\lambda_2\,\gamma_{n-1,2}(t),\nonumber\\
&& \hspace*{8cm} n=2,3,\ldots,\nonumber
\end{eqnarray}
with initial conditions
\begin{equation}
 \gamma_{n,1}(0)=p\,\delta_{n,j}, \qquad \gamma_{n,2}(0)=(1-p)\,\delta_{n,j} \qquad (0\leq p\leq 1).
 \label{eq:gammainit}
\end{equation}
\par
In order to obtain suitable relations for the state probabilities (\ref{absorb_prob_discr}), we 
recall that the transition probabilities avoiding state 0 for the $M/M/1$ queue 
with arrival rate $\lambda_i$ and service rate $\mu_i$, 
for $j,n\in \mathbb N$ and $t>0$ is (cf.\ Abate {\em et al.}\ \cite{AKW1991})
\begin{equation}
 \widehat \alpha_{j,n}^{(i)}(t)
 = e^{-(\lambda_i+\mu_i)t}\Big(\frac{\lambda_i}{\mu_i}\Big)^{(n-j)/2}
 \,\left[I_{n-j}(2t\sqrt{\lambda_i\mu_i})-I_{n+j}(2t\sqrt{\lambda_i\mu_i})\right]
  \label{MM1_abs_prob}
\end{equation}
and, due to relation $I_{j-1}(z)-I_{j+1}(z)=2j \,I_j(z)/z$ (cf.\ Eq.~8.486.1, p.\ 928 of \cite{GR2007}),  
\begin{equation}
 \widehat \alpha_{j,1}^{(i)}(t)
 ={j \over \mu_i\,t}\,e^{-(\lambda_i+\mu_i)t}\Big({\mu_i\over\lambda_i}\Big)^{j/2}
 I_j(2t\sqrt{\lambda_i\mu_i}).
 \label{MM1_abs_prob_1}
\end{equation}
%
\begin{proposition}\label{prop_gamma}
If $\eta_2=0$,  for $j\in \mathbb{N}$, $n\in \mathbb{N}_0$ and $t>0$, for the state probabilities (\ref{absorb_prob_discr}) 
one has:
\begin{eqnarray}
&&\hspace*{-1.0cm}\gamma_{n,1}(t)= p\, e^{-\eta_1 t} \,\widehat \alpha_{j,n}^{(1)}(t),
\label{eq:absprob_gen1}\\
&&\hspace*{-1.0cm}\gamma_{n,2}(t)=(1-p)\,  \widehat \alpha_{j,n}^{(2)}(t)
+  \eta_1 p \sum_{k=1}^{\infty}\int_0^t e^{-\eta_1 \tau} 
\,\widehat \alpha_{j,k}^{(1)}(\tau)\,\widehat \alpha_{k,n}^{(2)}(t-\tau)\; d\tau,
\label{eq:absprob_gen2}
\end{eqnarray}
where $\widehat \alpha_{j,k}^{(i)}(t)$ are given in  (\ref{MM1_abs_prob}) and (\ref{MM1_abs_prob_1}).
\end{proposition}
\begin{proof}
It follows from   (\ref{equat_genv1}) and (\ref{equat_genv2}), recalling the initial conditions (\ref{eq:gammainit}).\par\hfill\fine
\end{proof}
\par
The term in the right-hand-side of (\ref{eq:absprob_gen1}) can be interpreted as follows: starting from $(j,1)$ at time 0, 
with probability $p$, the process reaches $(n,1)$, with $n\in\mathbb{N}$, at time $t$ without crossing $(0,1)$ in the 
interval $(0,\tau)$, and  no switches occurred up to time  $t$.
Instead, the 2 terms in the right-hand-side of (\ref{eq:absprob_gen2}) have the following meaning: 
\begin{description}
\item{-} starting from $(j,2)$ at time 0, with probability $1-p$, the process reaches $(n,2)$ at time $t$, with 
$n\in\mathbb{N}$, at time $t$ without crossing $(0,2)$ in the interval $(0,\tau)$;
\item{-} starting from $(j,1)$ at time 0, with probability $p$, the process reaches $(k,1)$, with $k\in\mathbb{N}$, 
at time $\tau\in (0,t)$ without crossing $(0,1)$ in the interval $(0,\tau)$, then a switch occurs at  time  $t$, and 
starting from $(k,2)$ the process reaches $(n,2)$ at time $t$ without crossing $(0,2)$ in the interval $(\tau,t)$.
\end{description}
\par
For $j\in\mathbb{N}$, we consider the random variable 
$$
 T_j=\inf\{t>0:  {\bf N}(t)=(0,1)\;{\rm or}\;{\bf N}(t)=(0,2)\}, 
$$ 
where ${\bf N}(0)$ is given in (\ref{eq:initcondit_N}), with $j\in\mathbb{N}$. 
Note that $T_j$ denotes the  first-passage time (FPT) of the process ${\bf N}(t)$ 
through zero  starting from $(j,1)$ with probability $p$ 
and from $(j,2)$ with probability $1-p$, with $j\in \mathbb{N}$. Hence, $T_j$ is the first emptying time of the queue, 
with  the initial state specified in (\ref{eq:initcondit_N}). 
For $\eta_2=0$, the FPT of ${\bf N}(t)$ through $(0,1)$ or $(0,2)$ is identically 
distributed as the FPT of $\widetilde{\bf N}(t)$ through the same states. Therefore, the first passage through 
$(0,1)$ or $(0,2)$ for ${\bf N}(t)$ can be studied via the probabilities obtained in Proposition~\ref{prop_gamma}. 
Specifically, recalling (\ref{absorb_prob_discr}), for $j\in\mathbb{N}$ we have:
\begin{equation}
 \mathbb{P}(T_j<t)+\sum_{n=1}^{+\infty}\bigl[\gamma_{n,1}(t)+\gamma_{n,2}(t)\bigr]=1.
 \label{abs_first_passage}
\end{equation}
We focus our attention on the FPT probability density  
\begin{equation}
 b_j(t)={d\over dt}\, \mathbb{P}(T_j<t),\qquad t>0,\;j\in\mathbb{N}.
 \label{def_FPTdens_discr}
\end{equation}
Hereafter we show that such density can be expressed in terms of the first-passage-time density 
from state $j\in \mathbb N$ to state 0 of  the $M/M/1$ queues with rates $\lambda_i$ and $\mu_i$, 
$i=1,2$, given by (see, for instance, \cite{AKW1991}):
\begin{equation}
 \widehat g_{j,0}^{(i)} (t)=\frac{j}{t}\,e^{-(\lambda_i+\mu_i)t}
 \Big(\frac{\mu_i}{\lambda_i} \Big)^{j/2}\,I_j(2t\sqrt{\lambda_i\mu_i}),\qquad t>0.
 \label{FPT_density_MM1}
\end{equation}
\begin{proposition}
If $\eta_2=0$, for $t>0$ FPT probability density  (\ref{def_FPTdens_discr}) is expressed as: 
\begin{eqnarray}
&&b_j(t) = p\, e^{-\eta_1 t} \,\widehat g_{j,0}^{(1)}(t)+(1-p)\,  \widehat g_{j,0}^{(2)}(t)
\nonumber \\
&&\hspace*{1cm} +  \eta_1 p \sum_{k=1}^{\infty}\int_0^t e^{-\eta_1 \tau} 
\,\widehat \alpha_{j,k}^{(1)}(\tau)\,\widehat g_{k,0}^{(2)}(t-\tau)\; d\tau,\qquad j\in \mathbb{N},
\label{eq:BPdensity_gen}
\end{eqnarray}
where $\widehat g_{k,0}^{(i)}(t)$ is the FPT density given in (\ref{FPT_density_MM1}).  
\end{proposition}
\begin{proof} Making use of (\ref{equat_genv1}) and (\ref{equat_genv2}) one has
$$
{d\over dt} \sum_{n=1}^{+\infty}\gamma_{n,i}(t)
 =-\mu_i\gamma_{1,i}(t)+(-1)^i\eta_1\sum_{n=1}^{+\infty}\gamma_{n,1}(t),\qquad i=1,2,
$$
so that from (\ref{abs_first_passage}) and (\ref{def_FPTdens_discr}) we obtain 
\begin{equation}
 b_j(t)=\mu_1\,\gamma_{1,1}(t)+\mu_2\,\gamma_{1,2}(t),  \qquad t>0,\; j\in \mathbb{N}.
 \label{density_FPT_1}
\end{equation}
Due to Proposition~\ref{prop_gamma}, from (\ref{density_FPT_1}) for $j\in \mathbb{N}$
it follows:
\begin{eqnarray}
&& b_j(t)=\mu_1\,p\, e^{-\eta_1 t} \,\widehat \alpha_{j,1}^{(1)}(t)
+\mu_2\,(1-p)\,  \widehat \alpha_{j,1}^{(2)}(t)\nonumber\\
&&\hspace*{1.2cm}+  \mu_2\,\eta_1 p \sum_{k=1}^{\infty}\int_0^t e^{-\eta_1 \tau} 
\,\widehat \alpha_{j,k}^{(1)}(\tau)\,\widehat \alpha_{k,1}^{(2)}(t-\tau)\; d\tau,\qquad t>0.
\label{density_FPT_2}
\end{eqnarray}
By virtue of (\ref{MM1_abs_prob_1}) and (\ref{FPT_density_MM1})   one has $\mu_i\,\widehat \alpha_{j,1}^{(i)}(t)= \widehat g_{j,0}^{(i)}(t)$ for $i=1,2$, 
so that (\ref{eq:BPdensity_gen}) immediately follows from (\ref{density_FPT_2}).\hfill\fine
\end{proof}
For $j\in\mathbb{N}$, the 3 terms in the right-hand-side of the FPT density (\ref{eq:BPdensity_gen}) can be interpreted as follows:
\begin{description}
\item{-} starting from $(j,1)$ at time 0, with probability $p$, the process reaches $(0,1)$ for the first time at time $t$, 
and  no switches occurred up to time  $t$; 
\item{-} starting from $(j,2)$ at time 0, with probability $1-p$, the process reaches $(0,2)$ for the first time at time $t$;
\item{-} starting from $(j,1)$ at time 0, with probability $p$, the process reaches $(k,1)$, with $k\in\mathbb{N}$, 
at time $\tau\in (0,t)$ without crossing $(0,1)$ in the interval $(0,\tau)$, then a switch occurs at  time  $t$, and 
starting from $(k,2)$ the process reaches $(0,2)$ for the first time at time $t$.
\end{description}
\par
Let 
$$
B_j(s)  = {\cal L}[b_j(t)]=\int_0^{+\infty}e^{-s\,t}b_j(t)\;dt,\qquad s>0,j\in\mathbb{N}
$$
be the Laplace transform of the FPT density $b_j(t)$.
\begin{proposition} For $s>0$ and $j\in\mathbb{N}$, one has
\begin{eqnarray}
&&\hspace*{-1.0cm}B_j(s)  =  {p\over [\varphi_1(s)]^j} +{1-p\over [\psi_1(s)]^j} 
+{\eta_1\,p\,\varphi_2(s)\over\lambda_1[\psi_1(s)-\varphi_1(s)]\,[\psi_1(s)-\varphi_2(s)]}
\nonumber\\
&&\hspace*{0.5cm}\times {[\psi_1(s)]^j-[\varphi_1(s)]^j\over [\varphi_1(s)]^{j-1}[\psi_1(s)]^{j-1}},
\qquad s>0,\; j\in\mathbb{N},
\label{eq:LTBPdensity_gen}
\end{eqnarray}
with 
\begin{equation}
 \varphi_1(s),\varphi_2(s)=\frac{s+\lambda_1+\mu_1+\eta_1 \pm 
 \sqrt{  (s+ \lambda_1+\mu_1+\eta_1)^2-4\lambda_1\mu_1}}{2\mu_1}, 
 \label{eq:defphi_i}
\end{equation}
for $0<\varphi_2(s)<1<\varphi_1(s)$, and
\begin{equation}
 \psi_1(s), \psi_2(s)=\frac{s+\lambda_2+\mu_2 \pm 
 \sqrt{  (s+ \lambda_2+\mu_2)^2-4\lambda_2\mu_2}}{2\mu_2}, 
 \label{eq:defpsi_i}
\end{equation}
for $0<\psi_2(s)<1<\psi_1(s)$. 
\end{proposition}
\par
The Laplace transform (\ref{eq:LTBPdensity_gen}) is useful to calculate the probability that 
the process ${\bf N}(t)$ eventually reach the states $(0,1)$ or $(0,2)$ and the FPT  moments. 
\par
Let us now determine the  probability of the eventual first queue emptying. 
Since $\eta_2=0$, if $\lambda_2\leq \mu_2$ then  $\psi_1(0)=1$, so that 
${\mathbb P}(T_j<\infty)=1$, whereas if $\lambda_2> \mu_2$, then we have 
$\psi_1(0)=\lambda_2/\mu_2$, so that for $j\in\mathbb{N}$ it results:
\begin{eqnarray}
&&\hspace*{-1.4cm}\mathbb{P}(T_j<+\infty)= \int_0^{+\infty}b_j(t)\;dt=B_j(0)=
{p\over [\varphi_1(0)]^j} +(1-p)\Bigl({\mu_2\over\lambda_2}\Bigr)^j\nonumber\\
&&\hspace*{-0.5cm}+{\eta_1\,p\,\varphi_2(0)\over\lambda_1[\varphi_1(0)-\lambda_2/\mu_2][\varphi_2(0)-\lambda_2/\mu_2]}
 \Bigl({\mu_2\over\lambda_2}\Bigr)^{j-1}{(\lambda_2/\mu_2)^j-[\varphi_1(0)]^j\over [\varphi_1(0)]^{j-1}}\cdot
\label{eq:LTBPprob}
\end{eqnarray}
Note that if $j=1$, $b_1(t)$ represents the  busy period density of the considered queueing model; hence, if $\lambda_2\leq \mu_2$
the busy period termination is certain. On the contrary, if $\lambda_2> \mu_2$, (\ref{eq:LTBPprob}) takes the simplest form:
 $$
{\mathbb  P}(T_1<+\infty)=\int_0^{+\infty}b_1(t)\;dt=
 {p\over \varphi_1(0)}+{1-p\over\psi_1(0)}
 +{p\, \eta_1 \varphi_2(0)\over \lambda_1[\psi_1(0)-\varphi_2(0)]}.
$$
For $\lambda_2>\mu_2$, in Figure~\ref{fig10} we plot  $\mathbb{P}(T_j<+\infty)$, given in (\ref{eq:LTBPprob}), as function of $\eta_2$ for $j=1,3,5,10$. 
The case $j = 1$ corresponds to the probability that the busy period ends.
%
\begin{figure}[t]  
\centering
\subfigure[$\lambda_1=1.0,\mu_1=0.5$]{\includegraphics[width=0.42\textwidth]{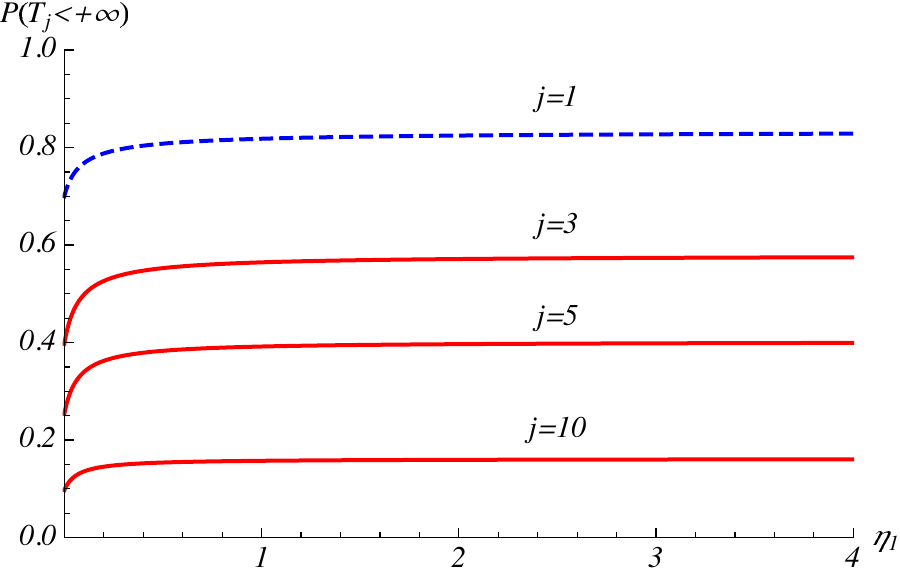}}
\hspace*{4mm}
\subfigure[$\lambda_1=0.5,\mu_1=1.0$]{\includegraphics[width=0.42\textwidth]{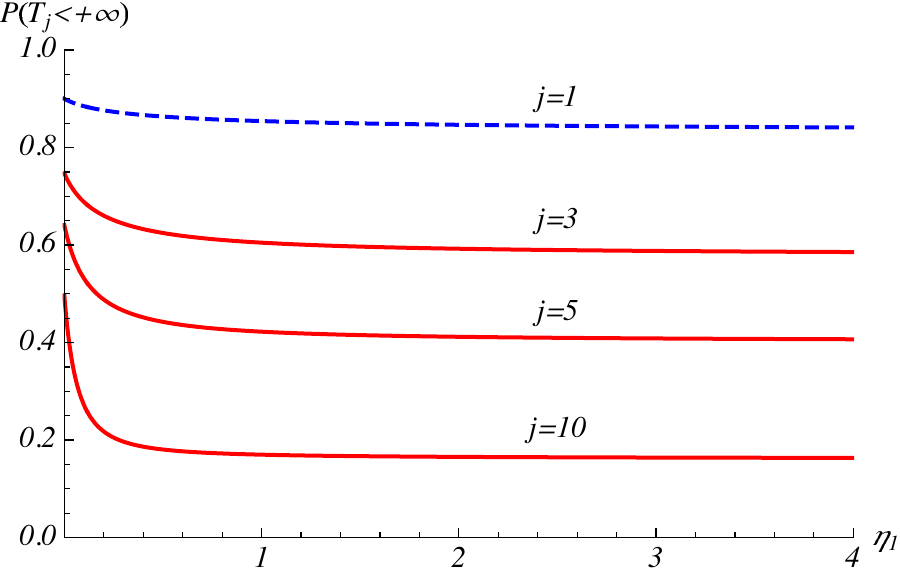}}\\
\caption{Plots of FPT probabilities $\mathbb{P}(T_j<+\infty)$ as function of $\eta_1$ for  $\eta_2=0$, $\lambda_2=1.2$, $\mu_2=1.0$ and  $p=0.4$.}
\label{fig10}
\end{figure}
%
\par
When $\lambda_2<\mu_2$, $\eta_1>0$ and $j\in\mathbb{N}$, from (\ref{eq:LTBPdensity_gen}) we obtain the FPT mean 
\begin{eqnarray}
&&\hspace*{-0.8cm}\mathbb{E}(T_j)={j\over \mu_2-\lambda_2}+p\,\biggl[1-{1\over [\varphi_1(0)]^j} \biggr]\,\biggl[ {1\over (\lambda_1-\mu_1+\eta_1)\varphi_1(0)-(\lambda_1-\mu_1-\eta_1)}\nonumber\\
&&\hspace*{0.8cm}+{1\over (\lambda_1-\mu_1+\eta_1)\varphi_2(0)-(\lambda_1-\mu_1-\eta_1)}-{1\over\eta_1}\,{\mu_1-\lambda_1\over\mu_2-\lambda_2}\biggr]\cdot
\label{FPT_mean_discret}
\end{eqnarray}
Clearly, for $p=0$, the right-hand side of (\ref{FPT_mean_discret}) corresponds to the FPT mean of the $M/M/1$ queue in the second environment.
\par
Finally, in Figure~\ref{fig11} we plot the mean (\ref{FPT_mean_discret}) as function of $\eta_1$; since $\lambda_2< \mu_2$, the first passage through zero state is a 
certain event. 
\begin{figure}[t]  
\centering
\subfigure[$\lambda_1=1.0,\mu_1=0.5$]{ \includegraphics[width=0.42\textwidth]{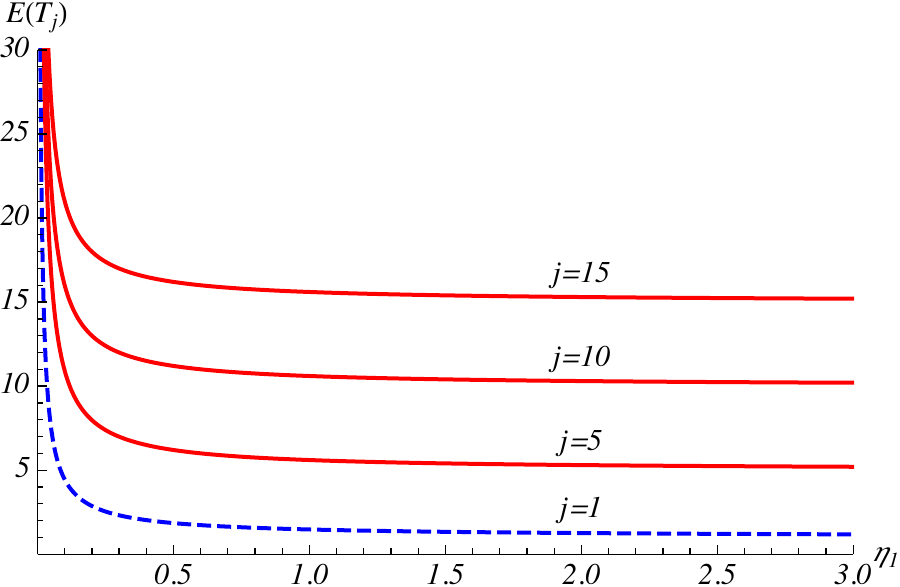}}
\hspace*{4mm}
\subfigure[$\lambda_1=0.5,\mu_1=1.0$]{ \includegraphics[width=0.42\textwidth]{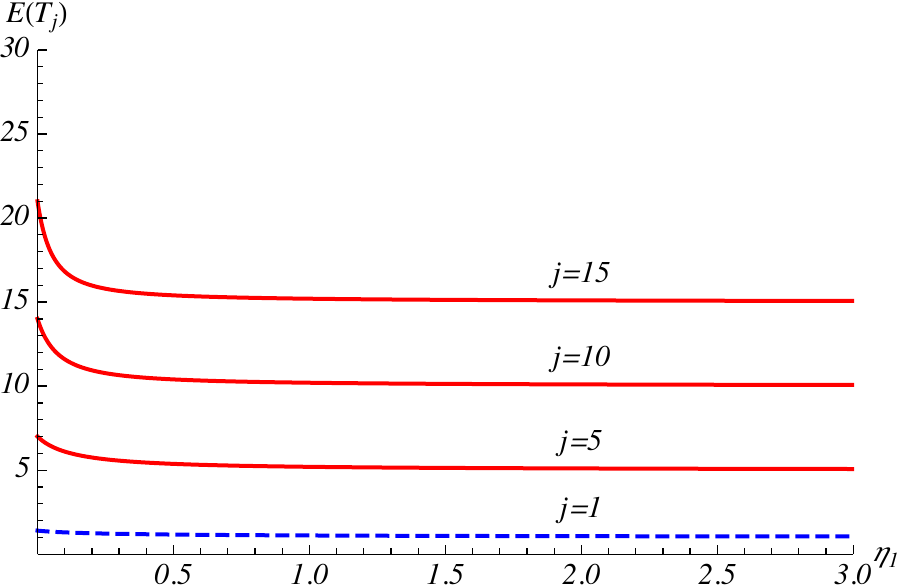}}\\
\caption{Plots of FPT mean $\mathbb{E}(T_j)$ as function of $\eta_1$ for  $\eta_2=0$, $\lambda_2=1.0$, $\mu_2=2.0$ and  $p=0.4$.}
\label{fig11}
\end{figure}
%
\section{Diffusion approximation}\label{section4}
This section is devoted to the construction of a heavy-traffic diffusion approximation for the process ${\bf N}(t)$. 
As customary, we adopt a scaling procedure that is usual in queueing theory and in other contexts  
(c.f.\ Di Crescenzo  {\em et al.}\ \cite{DGN2003} or Dharmaraja {\em et al.}\ \cite{DDGN2015}, 
for instance). As a first step, we perform a different parameterization of the arrival and service 
rates of the stochastic model introduced in Section \ref{section2}. Specifically, we set  
\begin{equation}
  \lambda_i=\frac{ \lambda_i^*}{\epsilon}+\frac{\omega_i^2}{2\epsilon^2}, 
  \qquad 
  \mu_i=\frac{\mu_i^*}{\epsilon}+\frac{\omega_i^2}{2\epsilon^2}
  \qquad  (i=1,2),
 \label{eq:rates}
\end{equation}
where $ \lambda_i^*>0$, $ \mu_i^*>0$, $\omega_i^2>0$, for $i=1,2$, and $\epsilon>0$. 
We remark that $\epsilon$ is a positive parameter that has a relevant role in the scaling  
procedure indicated below. 
\par
Let us now consider the position $N_\epsilon^*(t)=N(t)\epsilon$, for any $t>0$. Hence, 
the process $\{{\bf N}^*_{\epsilon}(t)=[N_\epsilon^*(t),\mathscr{E}(t)]=[N(t)\epsilon,\mathscr{E}(t)],t>0\}$
is a two-dimensional continuous-time Markov chain, having state-space 
$\mathbb{N}_{0,\epsilon}\times\{1,2\}$, where $\mathbb{N}_{0,\epsilon}=\{0, \epsilon, 2\epsilon, \ldots\}$. 
We denote the transient probabilities of ${\bf N}^*_{\epsilon}(t)$, $ t>0$, as 
\begin{equation}
 p_{\epsilon}(n, i; t)
 =\mathbb P[{\bf N}^*_{\epsilon}(t)=(n \epsilon, i)]
 =\mathbb P[n \epsilon\leq  N_\epsilon^*(t) < (n+1) \epsilon, \mathscr{E}(t)=i)], 
 \label{eq:apptransprob}
\end{equation}
for $n\in \mathbb{N}_0$ and $ i=1,2$. 
In the limit as $\epsilon \to 0^+$, it can be shown that the scaled process $N_\epsilon^*(t)$
converges weakly to a two-dimensional stochastic process, say 
$\{{\bf X}(t) =[X(t),\mathscr{E}(t)], t \geq 0\}$, having state-space $\mathbb R^+_0\times \{1,2\}$. 
Note that ${\bf X}(t)$ may be viewed as a restricted Wiener process alternating between two 
environments, with switching rates $\eta_1$ and $\eta_2$. 
For $x \in \mathbb R^+_0$, $t> 0$ and $i=1,2$, let 
\begin{equation}
 f_i(x,t)=\frac{d}{dx}\,\mathbb P\{ X(t) <x,\mathscr{E}(t) = i\} 
 \label{eq:densfjxt}
\end{equation}
denote the probability densities of the process ${\bf X}(t)$, where the initial state  
is
\begin{equation}
{\bf X}(0)=
\left\{\begin{array}{ll}
(y,1),&\;{\rm with\;probability}\; p,\\
(y,2),&\;{\rm with\;probability}\;1-p,
\end{array}\right.
 \label{eq:initcondit_X}
\end{equation}
with $y\in \mathbb{R}^+_0$. Starting from the forward 
equations for ${\bf N}(t)$, given in Eqs.\ (\ref{equat_env1}) and (\ref{equat_env2}), 
the scaling procedure mentioned above yields that the densities (\ref{eq:densfjxt}) satisfy 
the following partial differential equations of Kolmogorov type, for $i=1,2$, $x \in \mathbb R^+$ and $t> 0$: 
\begin{equation}
 \frac{\partial f_i(x,t)}{ \partial t}
 = -(\lambda_i^*-\mu_i^*)\frac{\partial f_i(x,t)}{ \partial x}
 +\frac{\omega_i^2}{2} \frac{\partial^2 f_i(x,t)}{ \partial x^2}
 +\eta_{3-i} f_{3-i}(x,t) - \eta_i f_i(x,t).
 \label{eqKolmogorov}
\end{equation}
We note that the first 2 terms in the right-hand-side of (\ref{eqKolmogorov}) correspond 
to the classical diffusive operators of a Wiener process, whereas the last 2 terms express the joking 
between the two different environments, occurring with switching rates $\eta_1$ and $\eta_2$. 
The first and second infinitesimal moments of the Wiener process in the $i$-th environment are 
respectively $\lambda_i^*-\mu_i^*$ and $\omega_i^2$, $i=1,2$. It is worth pointing out that, 
due to the scaling procedure, the first equations of systems (\ref{equat_env1}) and 
(\ref{equat_env2}) lead to the following reflecting condition at 0:
\begin{equation}
 \lim_{x\to 0^+} \left[(\lambda_i^*-\mu_i^*) f_i(x,t)-\frac{\omega_i^2}{2}\frac{\partial f_i(x,t)}{\partial x}\right]=0,
 \label{eq:reflecting}
\end{equation}
for $t>0$ and $i=1,2$. Moreover, since for the process ${\bf N}(t)$ the 
initial condition is expressed by a Bernoulli trial, similarly as (\ref{initial_condition})  
we have the following dichotomous initial condition for densities (\ref{eq:densfjxt}): 
\begin{equation}
\lim_{t\downarrow 0} f_1(x,t)=p\, \delta(x-y), \quad \lim_{t\downarrow 0}f_{2}(x,t)=(1-p)\,\delta(x-y),
 \qquad 0\leq p\leq 1,
 \label{initial_condition_cont}
\end{equation}
where $ \delta(\cdot)$ is the Dirac delta function. Furthermore, in analogy with  
(\ref{normalization_condition}), the normalization condition    
\begin{equation}
\int_{0}^{+\infty}[f_1(x,t)+f_2(x,t)]\; dx=1
\label{normalization_conditionf}
\end{equation}
holds for all $t\geq 0$. We remark that positions (\ref{eq:rates}) express a heavy-traffic condition, since 
the rates $\lambda_i$ and $\mu_i$ tend to  infinity when $\epsilon \to 0^+$ in the approximation procedure. 
%
\subsection{Steady-state density}
Let us now investigate the steady-state densities of ${\bf X}(t)$.  Let ${\bf X}=(X,\mathscr{E})$ be the two-dimensional random variable 
of the system in the steady-state regime. We aim to determine  the steady-state densities in the two environments, defined as  
\begin{equation}
W_i(x)=\lim_{t\to +\infty}f_i(x,t),\qquad x\in\mathbb{R}_0^+, \quad i=1,2.
\label{steady_state_density}
\end{equation}
From (\ref{eqKolmogorov}) and (\ref{eq:reflecting}) one has the following differential equations:
$$
-(\lambda_i^*-\mu_i^*){dW_i(x)\over dx}+{\omega_i^2\over 2} {d^2W_i(x)\over dx^2} +\eta_{3-i} f_{3-i}(x,t) - \eta_i f_i(x,t),\quad i=1,2,
$$
to be solved with the boundary conditions:
$$
\lim_{x\to 0^+} \left[(\lambda_i^*-\mu_i^*) W_i(x)- {\omega_i^2\over 2} {dW_i(x)\over dx}\right]=0,\quad i=1,2.
$$
Hence, denoting by 
$$
M_i(z)=\mathbb{E}[e^{zX}  \mathbbm{1}_{\mathscr{E}=i}]=\int_0^{+\infty}e^{zx}W_i(x)\;dx,\quad i=1,2
$$
the moment generating functions for the two environments in steady-state regime, one has:
\begin{eqnarray}
&&\hspace{-1.5cm}M_1(z)={2\eta_2\omega_2^2{\displaystyle\lim_{x\to 0}W_2(x)}-\omega_1^2[\omega_2^2z^2+2(\lambda_2^*-\mu_2^*)z-2\eta_2]{\displaystyle\lim_{x\to 0}}W_1(x)\over P^*(z)},\nonumber\\
&&\label{mom_generating_function_1}\\
&&\hspace{-1.5cm}M_2(z)={2\eta_1\omega_1^2{\displaystyle\lim_{x\to 0}}W_1(x)-\omega_2^2[\omega_1^2z^2+2(\lambda_1^*-\mu_1^*)z-2\eta_1]{\displaystyle\lim_{x\to 0}}W_2(x)\over P^*(z)},\nonumber
\end{eqnarray}
where $P^*(z)$ is the following third-degree polinomial in $z$:
\begin{eqnarray}
&&\hspace*{-1.0cm}P^*(z)=\omega_1^2\omega_2^2z^3+2[\omega_1^2(\lambda_2^*-\mu_2^*)+\omega_2^2(\lambda_1^*-\mu_1^*)]z^2\nonumber\\
&&\hspace*{0.5cm}-2[\omega_1^2\eta_2-2(\lambda_1^*-\mu_1^*)(\lambda_2^*-\mu_2^*)+\omega_2^2\eta_1]z\nonumber\\
&&\hspace*{0.5cm}-4[\eta_1(\lambda_2^*-\mu_2^*)+\eta_2(\lambda_1^*-\mu_1^*)].
\label{pol_continuo}
\end{eqnarray}
By taking into account the normalization condition $M_1(0)+M_2(0)=1$,  from (\ref{mom_generating_function_1}) one obtains:
\begin{equation}
\omega_1^2\lim_{x\to 0}W_1(x)+\omega_2^2\lim_{x\to 0}W_2(x)={2\,\bigl[\eta_1(\lambda_2^*-\mu_2^*)+\eta_2(\lambda_1^*-\mu_1^*)\bigr]\over \eta_1+\eta_2}\cdot
\label{cond_continue}
\end{equation}
Recalling that $\eta_1+\eta_2>0$, Eq.~(\ref{cond_continue}) shows that the steady-state regime 
exists if and only if one of the following cases holds:
\begin{description}
\item{\em (i)}  $\eta_2=0$ and $\lambda_2^*/\mu_2^*<1$, 
\item{\em (ii)}  $\eta_1=0$ and $\lambda_1^*/\mu_1^*<1$, 
\item{\em (iii)}   $\eta_1>0$, $\eta_2>0$ and  $\eta_1(\mu_2^*-\lambda_2^*)+ \eta_2 (\mu_1^*-\lambda_1^*)>0$. 
\end{description}
Hereafter, we  consider separately the  three cases.
\subsection*{$\bullet$ Case {\it (i)}}
If $\eta_2=0$ and $\lambda_2^*/\mu_2^*<1$, one can easily prove that
$$
W_1(x)=0, \qquad W_2(x)={2(\mu_2^*-\lambda_2^*)\over\omega_2^2}
\exp\Bigl\{-{2(\mu_2^*-\lambda_2^*)\,x\over\omega_2^2}\Bigr\},
\qquad x\in\mathbb{R}^+. 
$$
Hence, if $\eta_2=0$ and $\lambda_2^*/\mu_2^*<1$,  the steady-state density $W(x)=W_1(x)+W_2(x)$ is exponential, 
with parameter $2(\mu_2^*-\lambda_2^*)/\omega_2^2$. 
\subsection*{$\bullet$ Case {\it (ii)}}
If $\eta_1=0$ and $\lambda_1^*/\mu_1^*<1$, similarly to case {\em (i)}, one has
$$
W_1(x)={2(\mu_1^*-\lambda_1^*)\over\omega_1^2}\exp\Bigl\{-{2(\mu_1^*-\lambda_1^*)\,x\over\omega_1^2}\Bigr\},\qquad
W_2(x)=0, \qquad x\in\mathbb{R}^+,
$$
so that  the steady-state density $W(x)$ is exponential with parameter $2(\mu_1^*-\lambda_1^*)/\omega_1^2$. 
\subsection*{$\bullet$ Case {\it (iii)}}
Let $\eta_1>0$, $\eta_2>0$ and  $\eta_1(\mu_2^*-\lambda_2^*)+ \eta_2 (\mu_1^*-\lambda_1^*)>0$. Denoting by 
$\xi_1^*,\xi_2^*,\xi_3^*$ the roots of $P^*(z)$, given in (\ref{pol_continuo}), one has:
\begin{eqnarray}
&&\hspace*{-1.0cm}\xi_1^*+\xi_2^*+\xi_3^*={2\big[\omega_1^2(\mu_2^*-\lambda_2^*)+ \omega_2^2 (\mu_1^*-\lambda_1^*)\big]\over\omega_1^2\,\omega_2^2},\nonumber\\
&&\hspace*{-1.0cm}\xi_1^*\,\xi_2^*+\xi_1^*\,\xi_3^*+\xi_2^*\,\xi_3^*={-\big[\omega_1^2\,\eta_2-2(\mu_1^*-\lambda_1^*)(\mu_2^*-\lambda_2^*)+\omega_2^2\,\eta_1\big]
\over\omega_1^2\,\omega_2^2},
\label{roots1_cont}\\
&&\hspace*{-1.0cm}\xi_1^*\,\xi_2^*\,\xi_3^*={-4\bigl[\eta_1(\mu_2^*-\lambda_2^*)+ \eta_2 (\mu_1^*-\lambda_1^*)\bigr]\over\omega_1^2\,\omega_2^2},\nonumber
\end{eqnarray}
so that $\xi_1^*\,\xi_2^*\,\xi_3^*<0$. Furthermore, from  (\ref{pol_continuo}) it follows:
\begin{eqnarray}
&&P^*(0)=4\bigl[\eta_1(\mu_2^*-\lambda_2^*)+ \eta_2 (\mu_1^*-\lambda_1^*)\bigr]>0,\nonumber\\
&&P^*\Bigl({2(\mu_1^*-\lambda_1^*)\over\omega_1^2}\Bigr)={4\eta_1\big[\omega_1^2(\mu_2^*-\lambda_2^*)- \omega_2^2 (\mu_1^*-\lambda_1^*)\big]
\over\omega_1^2},\label{roots2_cont}\\
&&P^*\Bigl({2(\mu_2^*-\lambda_2^*)\over\omega_2^2}\Bigr)={-4\eta_2\big[\omega_1^2(\mu_2^*-\lambda_2^*)- \omega_2^2 (\mu_1^*-\lambda_1^*)\big]
\over\omega_2^2}.\nonumber
\end{eqnarray}
Making use of (\ref{roots1_cont}) and (\ref{roots2_cont}), it is not hard to prove that $P^*(z)$ has one negative root and two positive roots. In the sequel, we 
assume that $\xi_1^*>0, \xi_2^*>0$ and $\xi_3^*<0$, and $P^*(z)=\omega_1^2\omega_2^2(z-\xi_1^*)(z-\xi_2^*)(z-\xi_3^*)$. 
\begin{proposition}\label{prop_joint_dens_cont}
If $\eta_1>0$, $\eta_2>0$ and  $\eta_1(\mu_2^*-\lambda_2^*)+ \eta_2 (\mu_1^*-\lambda_1^*)>0$, then the  steady-state density of 
${\bf X}=(X,\mathscr{E})$ can be expressed in terms of the roots  $\xi_1^*>0, \xi_2^*>0$ and $\xi_3^*<0$ of the polynomial 
(\ref{pol_continuo}) as follows:
\begin{equation}
W_i(x)={\eta_{3-i}\over\eta_1+\eta_2}\,\Bigl[A_i^*h_1(x)+(1-A_i^*)h_2(x)\Bigr],\qquad x\in\mathbb{R}^+, \, i=1,2,
\label{steady_state_density_i}
\end{equation}
where $h_i(x)$ denotes an exponential density with mean $1/\xi_i^*$ and
\begin{equation}
A_i^*={4\,[\eta_1(\mu_2^*-\lambda_2^*)+\eta_2(\mu_1^*-\lambda_1^*)]\over\omega_1^2\,\omega_2^2\,\xi_1^*\,\xi_3^*\,(\xi_1^*-\xi_2^*)}\,
{\omega_{3-i}^2(\xi_1^*+\xi_3^*)-2(\mu_{3-i}^*-\lambda_{3-i}^*)\over \omega_{3-i}^2\xi_3^*-2(\mu_{3-i}^*-\lambda_{3-i}^*)} 
\label{coef_mixture_continuo}
\end{equation}
for $ i=1,2$.
\end{proposition}
\begin{proof}
Since $P^*(\xi_3^*)=0$, we require that  also the numerators of $M_1(z)$  and $M_2(z)$, given in (\ref{mom_generating_function_1}),  tend to zero 
as $z\to\xi_3^*$, so that by virtue of  (\ref{cond_continue}) one has:
\begin{eqnarray}
&&\hspace*{-1.0cm}\lim_{x\to 0}W_1(x)={4\eta_2\over\omega_1^2\,\xi_3^*[\omega_2^2\,\xi_3^*-2(\mu_2^*-\lambda_2^*)]}\,
{\eta_1(\mu_2^*-\lambda_2^*)+ \eta_2 (\mu_1^*-\lambda_1^*)\over\eta_1+\eta_2},\nonumber\\
\label{lim_density}\\
&&\hspace*{-1.0cm}\lim_{x\to 0}W_2(x)={4\eta_1\over\omega_2^2\,\xi_3^*[\omega_1^2\,\xi_3^*-2(\mu_1^*-\lambda_1^*)]}\,
{\eta_1(\mu_2^*-\lambda_2^*)+ \eta_2 (\mu_1^*-\lambda_1^*)\over\eta_1+\eta_2}\cdot\nonumber
\end{eqnarray}
Note that from (\ref{roots1_cont}) and (\ref{roots2_cont}) we have $\xi_3^*\neq 2(\mu_1^*-\lambda_1^*)/\omega_1^2$ and 
$\xi_3^*\neq 2(\mu_2^*-\lambda_2^*)/\omega_2^2$. Hence, substituting (\ref{lim_density}) in (\ref{mom_generating_function_1}) 
one obtains:
\begin{eqnarray}
&&\hspace*{-1.5cm}M_1(z)={4\,\eta_2\over\omega_1^2\,\omega_2^2\,\xi_3^*}\,
{\eta_1(\mu_2^*-\lambda_2^*)+ \eta_2 (\mu_1^*-\lambda_1^*)\over\eta_1+\eta_2}\,
{1\over \omega_2^2\,\xi_3^*-2(\mu_2^*-\lambda_2^*)}\nonumber\\
&&\hspace*{1.5cm}\times{-\omega_2^2\,z-\omega_2^2\,\xi_3^*+2(\mu_2^*-\lambda_2^*)\over (z-\xi_1^*)(z-\xi_2^*)},\nonumber\\
&&\label{mom_generating_function_11}\\
&&\hspace*{-1.5cm}M_2(z)={4\,\eta_1\over\omega_1^2\,\omega_2^2\,\xi_3^*}\,
{\eta_1(\mu_2^*-\lambda_2^*)+ \eta_2 (\mu_1^*-\lambda_1^*)\over\eta_1+\eta_2}\,
{1\over \omega_1^2\,\xi_3^*-2(\mu_1^*-\lambda_1^*)}\nonumber\\
&&\hspace*{1.5cm}\times{-\omega_1^2\,z-\omega_1^2\,\xi_3^*+2(\mu_1^*-\lambda_1^*)\over (z-\xi_1^*)(z-\xi_2^*)}.\nonumber
\end{eqnarray}
Since the functions (\ref{mom_generating_function_11}) are finite for all $z$ in some interval containing the origin, 
the moment generating functions $M_1(z)$ and $M_2(z)$ determine the probability densities $W_1(x)$ and $W_2(x)$.  Indeed, 
by inverting the moment generating functions, for $x\in\mathbb{R}^+$ one obtains:
\begin{eqnarray}
&&\hspace*{-0.6cm}W_1(x)={4\,\eta_2\over\omega_1^2\,\omega_2^2\,\xi_3^*\,(\xi_1^*-\xi_2^*)}\,{\eta_1(\mu_2^*-\lambda_2^*)+ \eta_2 (\mu_1^*-\lambda_1^*)\over\eta_1+\eta_2}
{1\over 2(\mu_2^*-\lambda_2^*)-\omega_2^2\,\xi_3^*}\nonumber\\
&&\hspace*{-0.3cm}\times\Bigl\{\bigl[\omega_2^2(\xi_2^*+\xi_3^*)-2(\mu_2^*-\lambda_2^*)\bigr]\,e^{-\xi_2^*x}
-\bigl[\omega_2^2(\xi_1^*+\xi_3^*)-2(\mu_2^*-\lambda_2^*)\bigr]\,e^{-\xi_1^*x}\Bigr\},\nonumber\\
&&\label{steady_state_cont}\\
&&\hspace*{-0.6cm}W_2(x)={4\,\eta_1\over\omega_1^2\,\omega_2^2\,\xi_3^*\,(\xi_1^*-\xi_2^*)}\,{\eta_1(\mu_2^*-\lambda_2^*)+ \eta_2 (\mu_1^*-\lambda_1^*)\over\eta_1+\eta_2}
{1\over 2(\mu_1^*-\lambda_1^*)-\omega_1^2\,\xi_3^*}\nonumber\\
&&\hspace*{-0.3cm}\times\Bigl\{\bigl[\omega_1^2(\xi_2^*+\xi_3^*)-2(\mu_1^*-\lambda_1^*)\bigr]e^{-\xi_2^*x}
-\bigl[\omega_1^2(\xi_1^*+\xi_3^*)-2(\mu_1^*-\lambda_1^*)\bigr]e^{-\xi_1^*x}\Bigr\},\nonumber
\end{eqnarray}
from which (\ref{steady_state_density_i}) immediately follows.
\hfill\fine
\end{proof}
\par
%
\begin{figure}[t]  
\centering
\includegraphics[scale=0.6]{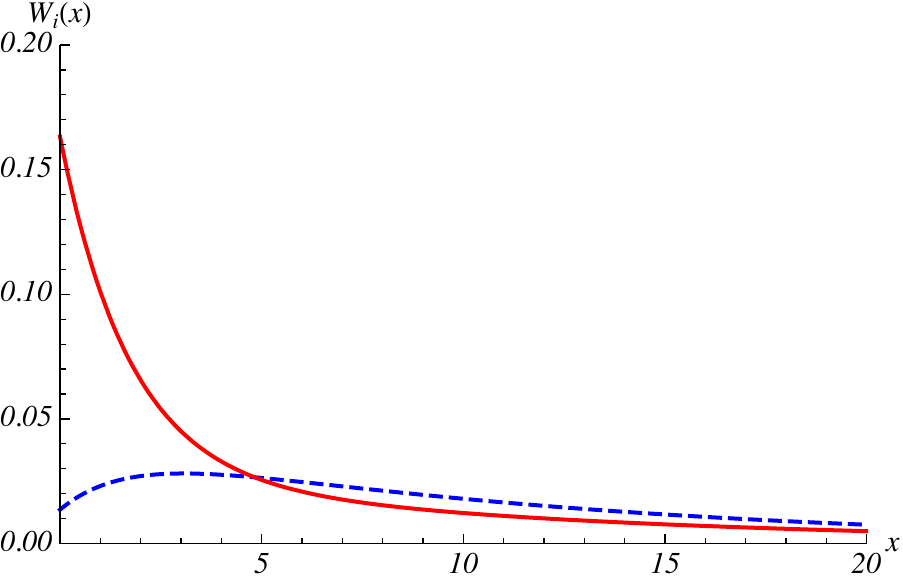}
\hspace*{4mm}
\includegraphics[scale=0.6]{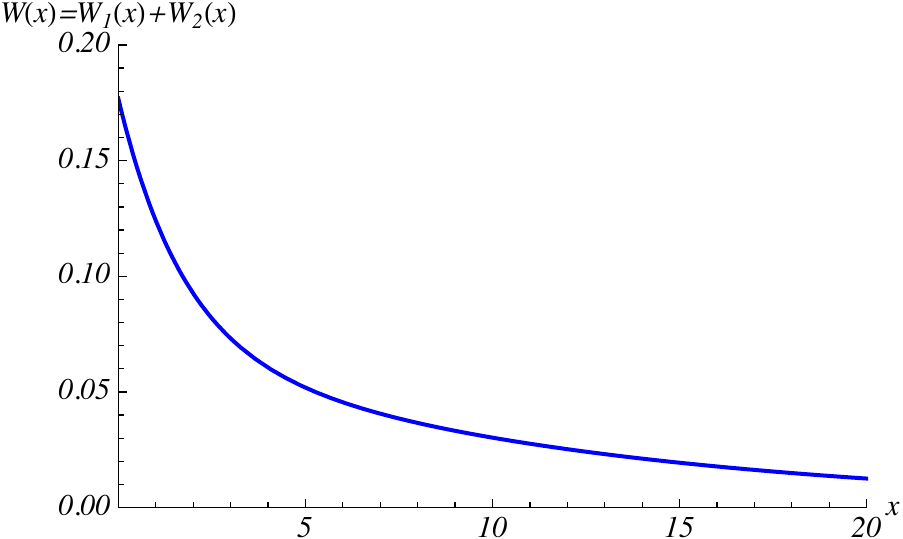}\\
\caption{Plots of densities $W_1(x)$ (line) and $W_2(x)$ (dashes), on the left, and 
$W(x)=W_1(x)+W_2(x)$, on the right, for  $\lambda_1^*=1$, $\mu_1^*=0.5$, $\lambda_2^*=1$, 
$\mu_2^*=2$, $\eta_1 = 0.1$, $\eta_2 = 0.08$, $\omega_1^2=1$ and $\omega_2^2=4$.}
\label{fig12}
\end{figure}
\par
In the continuous approximation, by virtue of (\ref{roots1_cont}), from  (\ref{mom_generating_function_11}) one has:
\begin{equation}
 \mathbb{P}(\mathscr{E}=i)\equiv M_i(0) =\int_0^{+\infty}W_i(x)\;dx={\eta_{3-i}\over \eta_1+\eta_2},\qquad i=1,2,
\label{eq:probE_cont}
\end{equation}
which provides the same result  given in (\ref{eq:probE}) for the discrete model. 
\par
Similarly to discrete case, we have that $\mathbb{P}(X<x|\mathscr{E}=1)$ and 
$\mathbb{P}(X<x|\mathscr{E}=2)$ are both generalized mixtures of two exponential distributions of means $1/\xi_1^*$ and $1/\xi_2^*$, respectively.
\par
Making use of Proposition~\ref{prop_joint_dens_cont} and of (\ref{eq:probE_cont}), the conditional means immediately follow:
\begin{equation}
\mathbb{E}[X|\mathscr{E}=i]=\int_{0}^{+\infty}x\,{W_i(x)\over \mathbb{P}(\mathscr{E}=i)}\;dx={A_i^*\over\xi_1^*}+{1-A_i^*\over\xi_2^*},\qquad i=1,2.
\label{cond_mean_cont}
\end{equation}
%
\begin{corollary} \label{corol2}
Under the assumptions of Proposition~\ref{prop_joint_dens_cont}, for $x\in\mathbb{R}^+$ one obtains the steady-state density 
of the process $X$:
\begin{equation} 
W(x)=W_1(x)+W_2(x)={\eta_2A_1^*+\eta_1A_2^*\over\eta_1+\eta_2}\,h_1(x)
+\Bigl[1-{\eta_2A_1^*+\eta_1A_2^*\over\eta_1+\eta_2}\Bigr]\,h_2(x),
\label{mixture_continue_system}
\end{equation}
where $A_1^*$ and $A_2^*$ are provided in (\ref{coef_mixture_continuo}) and $h_1(x)$, $h_2(x) $ are exponential density with means $1/\xi_1^*$ and $1/\xi_2^*$, respectively. 
\end{corollary}
Eq.~(\ref{mixture_continue_system}) shows that also $W(x)$ is a generalized mixture of two exponential densities with means 
 $1/\xi_1^*$ and $1/\xi_2^*$, respectively, so that
\begin{equation}
\mathbb{E}(X)={\eta_2A_1^*+\eta_1A_2^*\over\eta_1+\eta_2}{1\over \xi_1^*}+\Bigl[1-{\eta_2A_1^*+\eta_1A_2^*\over\eta_1+\eta_2}\Bigr]{1\over\xi_2^*}\cdot
 \label{expectations_continue_system}
\end{equation}
\par
Figure~\ref{fig12} shows the steady-state densities $W_1(x), W_2(x)$ (on the left) and 
$W(x)=W_1(x)+W_2(x)$ (on the right), obtained via Proposition \ref{prop_joint_dens_cont} and Corollary~\ref{corol2}, for 
$\lambda_1^*=1$, $\mu_1^*=0.5$, $\lambda_2^*=1$, $\mu_2^*=2$, $\eta_1 = 0.1$, $\eta_2 = 0.08$,
$\omega_1^2=1$ and  $\omega_2^2=4$. The roots of polynomial  (\ref{pol_continuo}) can be evaluated by means of 
MATHEMATICA$^{\footnotesize{\rm \textregistered}}$, so that 
$\xi_1^*=0.586811$, $\xi_2^*=0.0871$, $\xi_3^*=-1.17391$. 
%
\begin{figure}[t]  
\centering
\includegraphics[scale=0.6]{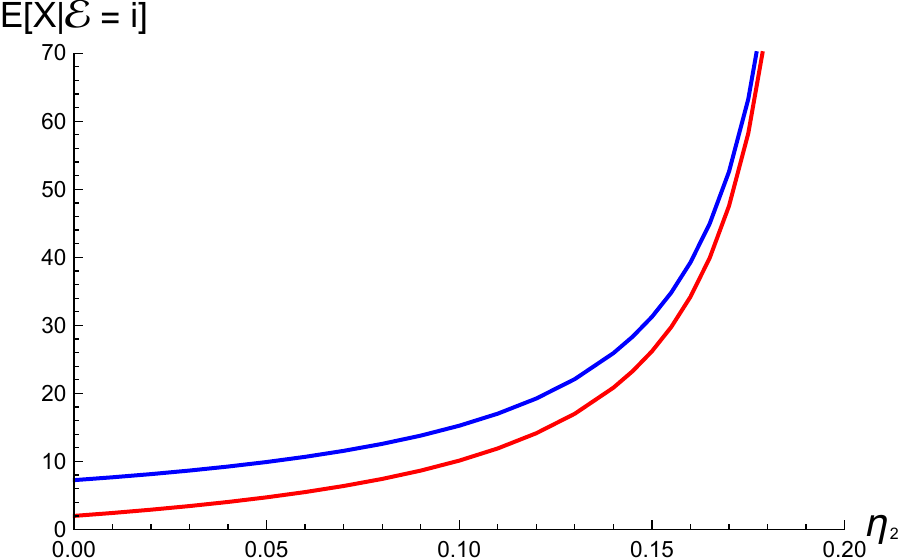}
\hspace*{4mm}
\includegraphics[scale=0.6]{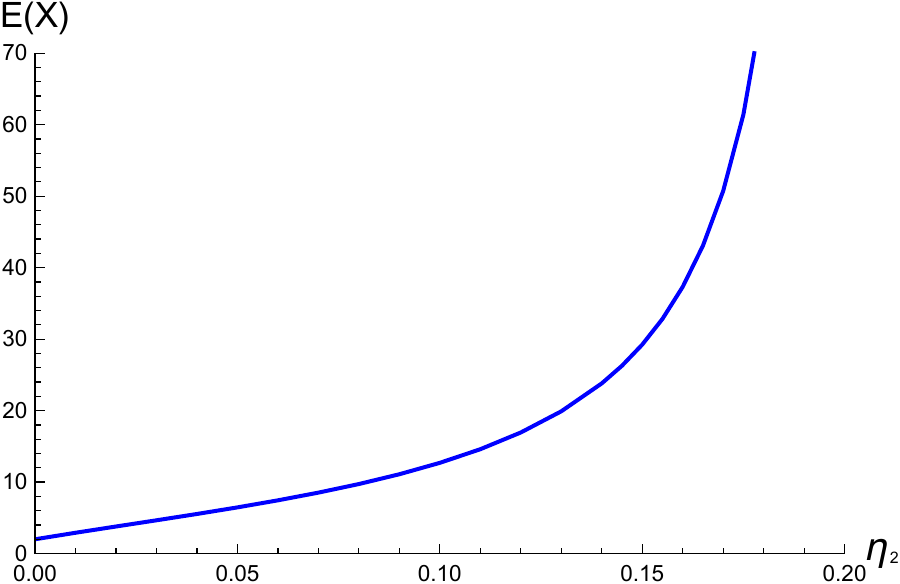}\\
\caption{For $\lambda_1^*=1$, $\mu_1^*=0.5$, $\lambda_2^*=1$, $\mu_2^*=2$, $\eta_1 = 0.1$, $\omega_1^2=1$, $\omega_2^2=4$
and $0\leq\eta_2 <0.2$, the  conditional means $\mathbb{E}[X|\mathscr{E}=i]$, given in  (\ref{cond_mean_cont}), are plotted on the left
for $i=1$ (top) and $i=2$ (bottom), whereas the mean $\mathbb{E}(X)$ is plotted on the right.
}
\label{fig13}
\end{figure} 
%
\begin{figure}[thb]  
\centering
\includegraphics[scale=0.6]{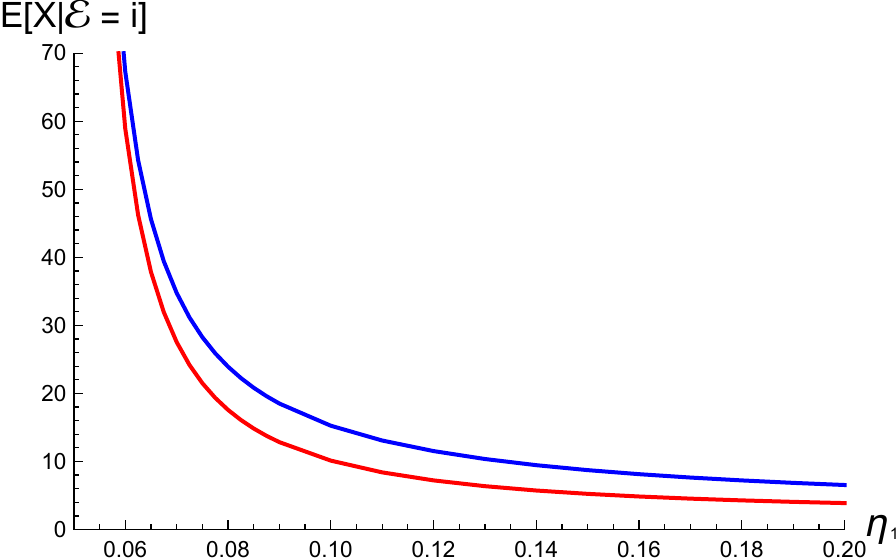}
\hspace*{4mm}
\includegraphics[scale=0.6]{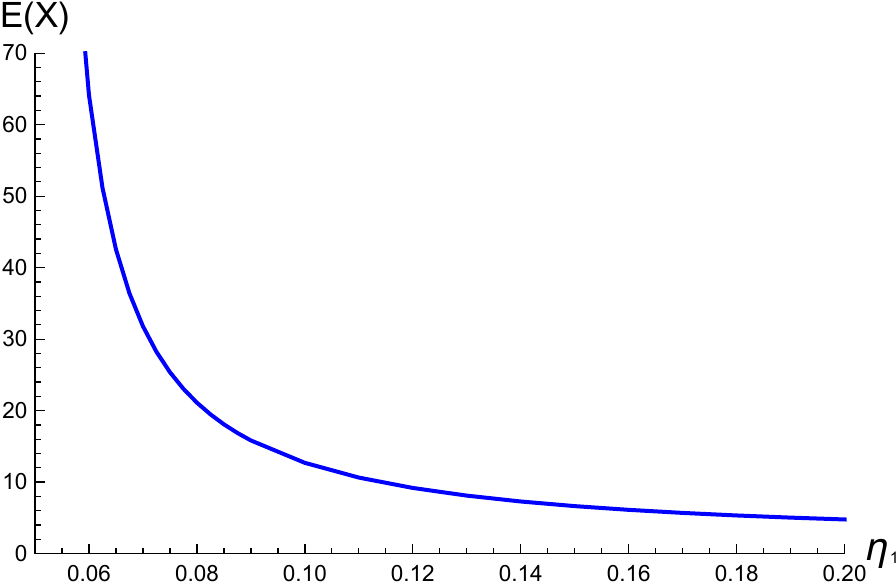}\\
\caption{For $\lambda_1^*=1$, $\mu_1^*=0.5$, $\lambda_2^*=1$, $\mu_2^*=2$, $\eta_2 = 0.1$, $\omega_1^2=1$, $\omega_2^2=4$
and $\eta_1>0.05$, the  conditional means $\mathbb{E}[X|\mathscr{E}=i]$, given in  (\ref{cond_mean_cont}), are plotted on the left
for $i=1$ (top) and $i=2$ (bottom), whereas the mean $\mathbb{E}(X)$ is plotted on the right.}
\label{fig14}
\end{figure} 
Figure~\ref{fig13} gives, on the left, a plot of the conditional means, obtained in (\ref{cond_mean_cont}),  
for a suitable choice of the parameters, showing that $\mathbb{E}[X|\mathscr{E}=i]$ is increasing in $\eta_2$, for $i=1,2$. Furthermore, 
on the right of Figure~\ref{fig13} is plotted $\mathbb{E}(X)$ for the same choices of parameters. Similarly, in Figure~\ref{fig14} 
$\mathbb{E}[X|\mathscr{E}=i]$ (on the left) and $\mathbb{E}(X)$ (on the left) are plotted for the same 
choices of parameters, showing that they are decreasing in $\eta_1$.  
\par
\begin{figure}[thb]  
\centering
\includegraphics[scale=0.6]{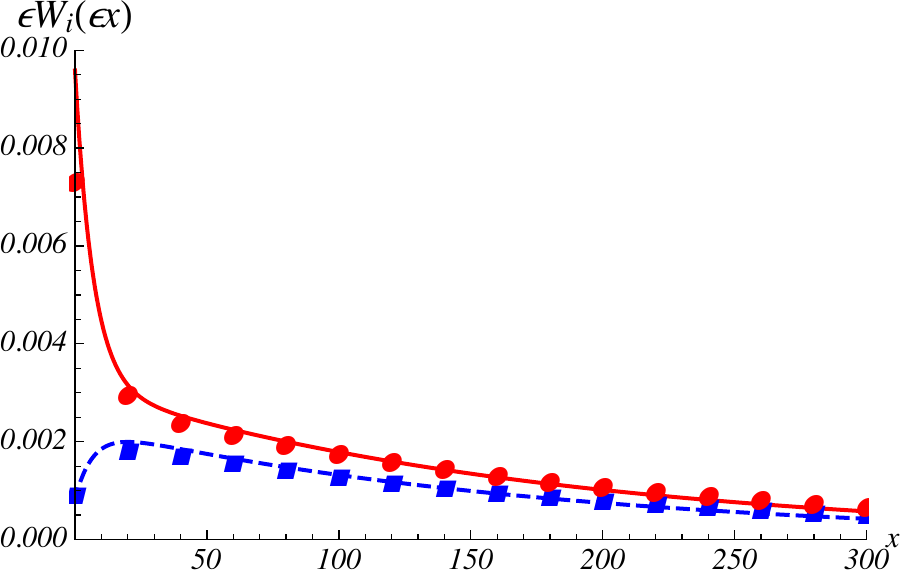}
\hspace*{4mm}
\includegraphics[scale=0.6]{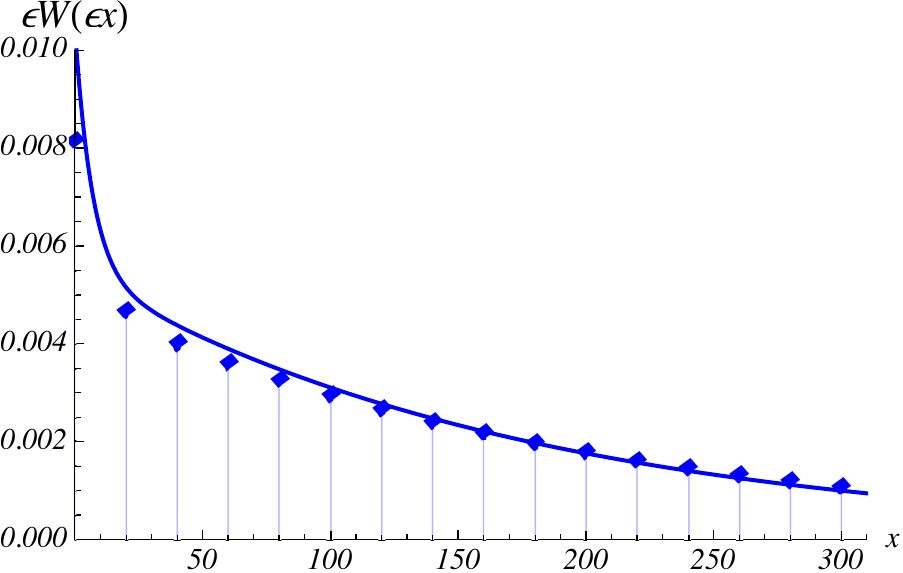}\\
\caption{For $\lambda_1^*=1$, $\mu_1^*=0.5$, $\lambda_2^*=0.8$, $\mu_2^*=1.2$,  $\eta_1 = 0.6$, $\eta_2 = 0.4$, $\omega_1^2=0.2$, $\omega_2^2=0.4$, on the left 
the  functions $\epsilon\,W_i(\epsilon x)$ are compared with the probabilities $q_{n,i}$ for $i=1$ (square) and $i=2$ (circle), where $\lambda_1,\mu_1,\lambda_2,\mu_2$ 
 are given in (\ref{eq:rates}) with $\epsilon=0.05$. On the right, for the same choices, the  function $\epsilon\,W(\epsilon x)$ is compared with the probabilities $q_n$  (diamond).}
\label{fig15}
\end{figure} 
%
\begin{figure}[thb]  
\centering
\includegraphics[scale=0.6]{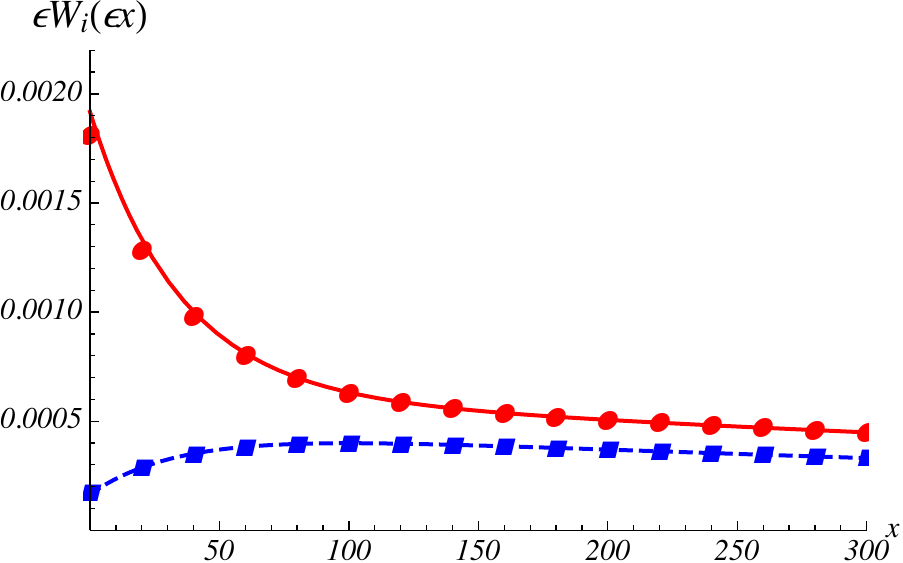}
\hspace*{4mm}
\includegraphics[scale=0.6]{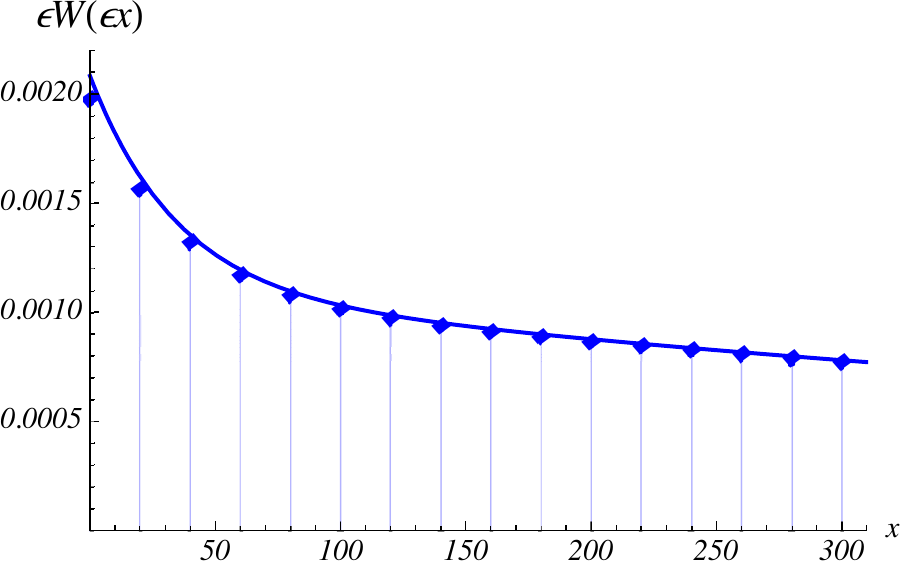}\\
\caption{As in Figure~\ref{fig15}, with $\epsilon=0.01$.}
\label{fig16}
\end{figure} 
\par
To show the validity tof the approximating procedure given in Section~\ref{section4},  in Figures~\ref{fig15}(a) and \ref{fig16}(a),  
we compare the  functions $\epsilon\,W_i(\epsilon x)$, where $W_i(x)$ is given in  
(\ref{steady_state_density_i}), with the probabilities $q_{n,i}$, given in  (\ref{mixture_discrete_environments}),  for $\epsilon=0.05$ and $\epsilon=0.01$, 
respectively. Furthermore, in Figures~\ref{fig15}(b) and \ref{fig16}(b),  we compare the  functions $\epsilon\,W(\epsilon x)
=\epsilon\,W_1(\epsilon x)+\epsilon\,W_2(\epsilon x)$, where $W(x)$ is given in  (\ref{mixture_continue_system}), with the probabilities $q_n$, given in  
(\ref{mixture_discrete_system}),  for $\epsilon=0.05$ and $\epsilon=0.01$, respectively. The probabilities $q_{n,1}$ (square), the probabilities $q_{n,1}$ (circle) 
and the probabilities $q_n$ (diamond) are represented for $n=20\,k$ $(k=0,1,\ldots,15)$. According to (\ref{eq:rates}), in  Figures~\ref{fig15} we set   
$\lambda_1=60$, $\mu_1=50$, $\lambda_2=96$, $\mu_2=104$, whereas in  Figures~\ref{fig16} one has 
$\lambda_1=1100$, $\mu_1=1050$, $\lambda_2=2080$, $\mu_2=2120$. From Figures~\ref{fig15} and \ref{fig16}, we note that 
the goodness of the diffusion approximation for the steady-state probabilities improves as $\epsilon$ decreases, due to an increase of traffic in the queueing system. 
%
\section{Analysis of the  diffusion process for $\eta_2=0$}\label{section5}
Let us now analyze the transient behaviour of the process ${\bf X}(t)=[X(t),\mathscr{E}(t)]$ 
in the case $\eta_2=0$, with the initial state specified in (\ref{eq:initcondit_X}).
\subsection{Probability densities}
Similarly as in the discrete model, hereafter we express the probability densities (\ref{eq:densfjxt}) 
in terms of the transition densities $\widehat{r}^{(i)}(x,t|y)$ of two Wiener processes $\widehat X^{(i)}(t)$, characterized by drift $\beta_i=\lambda_i^*-\mu_i^*$ and infinitesimal 
variance $\omega_i^2$, $i=1,2$, restricted to $[0,+\infty)$, with $0$ reflecting boundary, given by (cf.\ \cite{CoxMiller_1970}) 
\begin{eqnarray}
&&\hspace*{-2.3cm}\widehat{r}^{(i)}(x,t|y)={1\over\sqrt{2\pi \omega_i^2\,t}}\Biggl[ \exp\biggl\{-{(x-y-\beta_i t)^2\over 2\,\omega_i^2\, t}\biggr\}
  \nonumber\\
&&\hspace*{-0.6cm}
 +\exp\biggl\{-{2\beta_i\,y\over\omega_i^2}\biggr\}
 \exp\biggl\{ -{(x+y-\beta_i t)^2\over 2\,\omega_i^2\, t}\biggr\}\Biggr]
   \nonumber\\
&&\hspace*{-0.6cm} -{\beta_i\over\omega_i^2}\,\exp\biggl\{  {2\,\beta_i\, x\over\omega_i^2} \biggr\}\,{\rm Erfc} 
 \biggl( {x+y+\beta_i t\over \sqrt{2\, \omega_i^2\,t}}\biggr), \qquad x,y\in\mathbb{R}_0^+,
\label{eq:densita_Wiener}
\end{eqnarray}
where ${\rm Erfc}(x)=(2/\sqrt{\pi})\int_x^{+\infty}e^{-z^2}\,dz$ denotes the complementary error function. 
\begin{proposition}
Let $\eta_2=0$. For all $t\geq 0$ and $y,x\in\mathbb{R}^+_0$, the probability densities (\ref{eq:densfjxt}) satisfy
\begin{eqnarray}
&&\hspace*{-1.0cm} f_1(x,t)  =  p \,e^{-\eta_1 t}\, \widehat r^{(1)}(x,t\,|\,y),
  \label{eq:reldensf1}
\\
&&\hspace*{-1.0cm}f_2(x,t) =  (1-p)\, \widehat r^{(2)}(x,t\,|\,y)
 \nonumber \\
 &&\hspace*{0.5cm} +  p\,\eta_1 \int_0^{+\infty} dz \int_0^t \widehat r^{(1)}(z,\tau\,|\,y)\,  e^{-\eta_1\tau}\,
 \widehat r^{(2)}(x,t-\tau\,|\,z)\; d\tau,
 \label{eq:reldensf2}
\end{eqnarray}
where $\widehat r^{(i)}(x,t|y)$ are provided in (\ref{eq:densita_Wiener}). 
\end{proposition}
\begin{proof} It follows from  (\ref{eqKolmogorov}), taking into account the boundary conditions (\ref{eq:reflecting}) and 
the initial conditions (\ref{initial_condition_cont}).\hfill\fine
\end{proof}
The probabilistic interpretation of (\ref{eq:reldensf1}) and (\ref{eq:reldensf2}) is similar to the discrete queueing model.
%
\subsection{First-passage time problem}
We  consider the first-passage time through $(0,1)$ or $(0,2)$ states when $\eta_2=0$. To this purpose, we define a  two-dimensional 
stochastic process $\{ \widetilde {\bf X}(t)=[\widetilde X(t),\mathscr{E}(t)],t\geq 0\}$,  obtained from ${\bf X}(t)$ by removing all the 
transitions from  $(0,1)$ and $(0,2)$. We assume that $\widetilde {\bf X}(0)=(y,1)$ with probability $p$ and $\widetilde {\bf X}(0)=(y,2)$ 
with probability $1-p$, being $y\in\mathbb{R}^+$. 
Similarly to the discrete queueing model, only transitions from the first to the second environment are allowed. 
Hence, for $y\in\mathbb{R}^+$, denoting by
\begin{equation}
 h_i(x,t|y)={d\over dt}\mathbb{P}\{\widetilde X(t)<x,\mathscr{E}(t)=i\},
 \qquad x\in\mathbb{R}^+_0,\; i=1,2, \; t\geq 0
 \label{eq:denshixt}
\end{equation}
the transition density of the process $\widetilde {\bf X}(t)$, one has:
\begin{eqnarray}
&&\hspace*{-1.5cm}{\partial h_1(x,t|y)\over \partial t}
 = -(\lambda_1^*-\mu_1^*){\partial h_1(x,t|y)\over \partial x}
 +{\omega_1^2\over 2} {\partial^2 h_1(x,t|y)\over \partial x^2}- \eta_1 h_1(x,t|y),\nonumber\\
&& \label{eq_abs_cont}\\
 &&\hspace*{-1.5cm}{\partial h_2(x,t|y)\over \partial t}
 = -(\lambda_2^*-\mu_2^*){\partial h_2(x,t|y)\over \partial x}
 +{\omega_2^2\over 2} {\partial^2 h_2(x,t|y)\over  \partial x^2}
+\eta_1 h_1(x,t|y),\nonumber
\end{eqnarray}
with the absorbing boundary conditions 
\begin{equation}
\lim_{x\downarrow 0} h_i(x,t|y)=0,\qquad i=1,2
 \label{bound_condition_abs_cont}
\end{equation}
and the initial conditions 
\begin{equation}
 \lim_{t\downarrow 0}h_1(x,t|y)=p\, \delta(x-y), \quad \lim_{t\downarrow 0}h_{2}(x,t|y)
 =(1-p)\,\delta(x-y), \qquad 0\leq p\leq 1.
 \label{initial_condition_abs_cont}
\end{equation}
\par
Hereafter we express the transition densities (\ref{eq:denshixt}) in terms of the probability densities of the Wiener processes  $\widehat X^{(i)}(t)$
in the presence of an absorbing boundary in the zero state for $x,y\in\mathbb{R}^+$, which is given by (cf.\ \cite{CoxMiller_1970}) 
\begin{eqnarray}
 && \hspace{-1cm} \widehat \alpha^{(i)}(x,t|y)={1\over\sqrt{2\,\pi\,\omega_i^2\, t}}\Biggl[ \exp\biggl\{-{(x-y-\beta_i\, t)^2\over 2\,\omega_i^2\, t}\biggr\}
 \nonumber \\
  && \hspace{1cm}
 -\exp\biggl\{-{2\beta_i\,y\over\omega_i^2}\biggr\}
 \exp\biggl\{- {(x+y-\beta_i\, t)^2\over 2\,\omega_i^2\, t}\biggr\}\Biggr], \qquad t>0.
\label{abs_dens_Wiener}
\end{eqnarray} 
%
\begin{proposition}\label{prop_dens_abs_cont}
If $\eta_2=0$, for $y\in\mathbb{R}^+$, $x\in\mathbb{R}^+_0$ and $t>0$, the transition densities (\ref{eq:denshixt}) can be expressed as:
\begin{eqnarray}
&&\hspace*{-1.5cm}h_1(x,t|y)=p\, e^{-\eta_1 t} \,\widehat \alpha^{(1)}(x,t|y),
\label{eq:absdens_gen1}\\
&&\hspace*{-1.5cm}h_2(x,t|y)=(1-p)\,  \widehat \alpha^{(2)}(x,t|y)\nonumber\\
&&+  \eta_1 p \int_0^{\infty}dz\int_0^t e^{-\eta_1 \tau} 
\,\widehat \alpha^{(1)}(z,\tau|y)\,\widehat \alpha^{(2)}(x,t-\tau|z)\; d\tau,
\label{eq:absdens_gen2}
\end{eqnarray}
where $\widehat \alpha^{(i)}(x,t|y)$ are provided in (\ref{abs_dens_Wiener}).
\end{proposition}
\begin{proof}
It follows from  (\ref{eq_abs_cont}), taking into account    the absorbing boundary conditions (\ref{bound_condition_abs_cont}) 
and the initial conditions (\ref{initial_condition_abs_cont}).\hfill\fine
\end{proof}
We note that Eqs.~(\ref{eq:absdens_gen1}) and (\ref{eq:absdens_gen2}) 
are similar to (\ref{eq:absprob_gen1}) and (\ref{eq:absprob_gen2}) for the discrete queueing model.
\par
For $y\in\mathbb{R}^+$, let 
$$
 {\cal T}_y=\inf\{t>0:  {\bf X}(t)=(0,1)\;{\rm or}\; {\bf X}(t)=(0,2)\}, 
$$ 
be the  FPT through zero  for ${\bf X}(t)$ starting from $(y,1)$ with probability $p$ 
and from $(y,2)$ with probability $1-p$. We note that
\begin{equation}
 \mathbb{P}({\cal T}_y<t)+\int_0^{+\infty}\bigl[ h_1(x,t|y)+h_2(x,t|y)\bigr]\;dx=1.
\label{abs_first_passage_cont}
\end{equation}
Hereafter we focus on the FPT probability density   
\begin{equation}
k(0,t|y)={d\over dt}\mathbb{P}({\cal T}_y<t),\qquad t>0,\; y\in\mathbb{R}^+.
\label{def_FPT_dens_cont}
\end{equation}
Specifically,  we express such density in terms of the  FPT densities from state $y$ to state $x$ 
for the Wiener processes $\widehat X^{(i)}(t)$, given by 
\begin{equation}
 \widehat g^{(i)}(x,t|y)={y-x\over\sqrt{2\,\pi\,\omega_i^2\, t^3}} \exp\biggl\{-{(x-y-\beta_i\, t)^2\over 2\,\omega_i^2\, t}\biggr\}, 
 \qquad 0\leq x<y.
 \label{FPT_density_Wiener}
\end{equation}
%
\begin{proposition}
If $\eta_2=0$ and $y\in\mathbb{R}^+$, for $t>0$ the FPT density (\ref{def_FPT_dens_cont}) can be 
expressed as  
\begin{eqnarray}
&&\hspace*{-1.8cm}k(0,t|y)=p\, e^{-\eta_1 t} \,\widehat g^{(1)}(0,t|y)+(1-p)\,  \widehat g^{(2)}(0,t|y)\nonumber\\
&&+  \eta_1 p \int_0^{+\infty}dz\int_0^t e^{-\eta_1 \tau} 
\,\widehat \alpha^{(1)}(z,\tau|y)\,\widehat g^{(2)}(0,t-\tau|z)\; d\tau,
\label{density_FPT_cont}
\end{eqnarray}
where $g^{(i)}(0,t|y)$ are provided in (\ref{FPT_density_Wiener}).
\end{proposition}
\begin{proof} Making use of (\ref{eq_abs_cont}), (\ref{bound_condition_abs_cont}) and (\ref{initial_condition_abs_cont}), for $i=1,2$ one has
$$
{d\over dt}\int_0^{+\infty} h_i(x,t|y)\; dx=-{\omega_i^2\over 2}\lim_{x\downarrow 0} {\partial\over\partial x}h_i(x,t|y)+(-1)^i\eta_i\int_0^{+\infty}
h_1(x,t|y)\;dx,
$$
so that, from (\ref{abs_first_passage_cont}) and  (\ref{def_FPT_dens_cont}) it follows:
\begin{equation}
k(0,t|y)={\omega_1^2\over 2}\lim_{x\downarrow 0}{\partial h_1(x,t|y)\over\partial x}+{\omega_2^2\over 2}\lim_{x\downarrow 0}{\partial h_2(x,t|y)\over\partial x}\cdot
 \label{density_FPT_1_cont}
\end{equation}
Recalling (\ref{abs_dens_Wiener}) and (\ref{FPT_density_Wiener}), for $y\in\mathbb{R}^+$ one has
$$
\lim_{x\downarrow 0}{\partial\over \partial x}\widehat \alpha^{(i)}(x,t|y)=2\,\omega_i^2\,\widehat g^{(i)}(0,t|y),\qquad i=1,2.
$$
Hence, (\ref{density_FPT_cont}) immediately follows from (\ref{eq:absdens_gen1}), (\ref{eq:absdens_gen2}) and
(\ref{density_FPT_1_cont}).\hfill\fine
\end{proof}
We note the high analogy in the FPT density $k(0,t|y)$, given in (\ref{density_FPT_cont}), with  the FPT density $b_j(t)$, 
given in (\ref{eq:BPdensity_gen}), of the discrete queueing model.
\par
Let 
$$
K(s|y)  = {\cal L}[k(0,t|y)]=\int_0^{+\infty}e^{-s\,t}k(0,t|y)\;dt,\qquad s>0,y\in\mathbb{R}^+
$$
be the Laplace transform of the FPT density $k(0,t|y)$.
%
\begin{proposition} For $s>0$ and $y\in\mathbb{R}^+$, one has
\begin{equation}
K(s|y)=p\,e^{-y\zeta_1(s)}+(1-p)\,e^{-y\theta_1(s)}+{2\,\eta_1\,p\,\bigl[e^{-y\zeta_1(s)}-e^{-y\theta_1(s)}\bigr]\over\omega_1^2
\bigl[\zeta_1(s)-\theta_1(s)\bigr]\,\bigl[\zeta_2(s)-\theta_1(s)\bigr]},
\label{eq:LTBPdensity_cont}
\end{equation}
with 
\begin{equation}
\zeta_1(s),\zeta_2(s)={\lambda_1^*-\mu_1^*\pm\sqrt{(\lambda_1^*-\mu_1^*)^2+2\,\omega_1^2\,(s+\eta_1)}\over\omega_1^2}
\label{zeta12}
\end{equation}
for $\zeta_2(s)<0<\zeta_1(s)$, and
\begin{equation}
\theta_1(s),\theta_2(s)={\lambda_2^*-\mu_2^*\pm\sqrt{(\lambda_2^*-\mu_2^*)^2+2\,\omega_2^2\,s}\over\omega_2^2}
\label{theta12}
\end{equation}
for $\theta_2(s)<0<\theta_1(s)$.
\end{proposition}
\par
From (\ref{eq:LTBPdensity_cont}) we determine the ultimate absorbing probability in  $(0,1)$ or in $(0,2)$. 
Since $\eta_2=0$, if $\lambda_2^*\leq \mu_2^*$ then  $\theta_1(0)=0$, so that 
${\mathbb P}({\cal T}_y<\infty)=1$, whereas if $\lambda^*_2> \mu^*_2$, then we have 
$\theta_1(0)=2(\lambda_2^*-\mu_2^*)/\omega_2^2$, so that for $y\in\mathbb{R}^+$ it results:
\begin{eqnarray}
&&\hspace*{-0.4cm}\mathbb{P}({\cal T}_y<+\infty)= \int_0^{+\infty}k(0,t|y)\;dt=p\,e^{-y\zeta_1(0)}+(1-p)\,\exp\Bigl\{-{2(\lambda_2^*-\mu_2^*)y\over\omega_2^2}\Bigr\}\nonumber\\
&&\hspace*{2.2cm}+{2\,\eta_1\,p\,\Bigl[e^{-y\zeta_1(0)}-\exp\Bigl\{-{2(\lambda_2^*-\mu_2^*)y\over\omega_2^2}\Bigr\}\Bigr]\over\omega_1^2
\Bigl[\zeta_1(0)-{2(\lambda_2^*-\mu_2^*)\over\omega_2^2}\Bigr]\,\Bigl[\zeta_2(0)-{2(\lambda_2^*-\mu_2^*)\over\omega_2^2}\Bigr]},
\label{eq:LTBPprob_cont}
\end{eqnarray}
with $\zeta_1(0),\zeta_2(0)$ given in (\ref{zeta12}) for $s=0$. 
%
\begin{figure}[t]  
\centering
\subfigure[$\lambda_1^*=1.0,\mu_1^*=0.5$]{\includegraphics[width=0.42\textwidth]{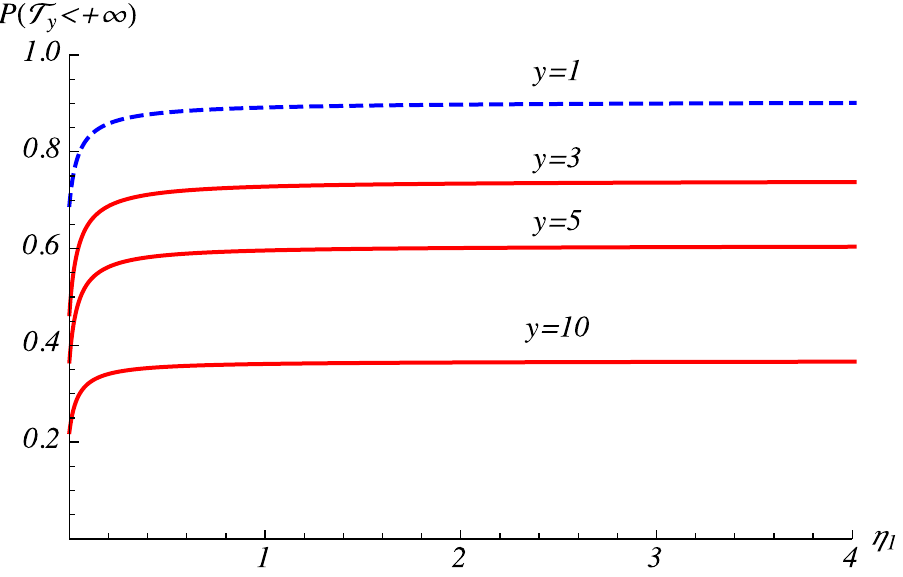}}
\hspace*{4mm}
\subfigure[$\lambda_1^*=0.5,\mu_1^*=1.0$]{\includegraphics[width=0.42\textwidth]{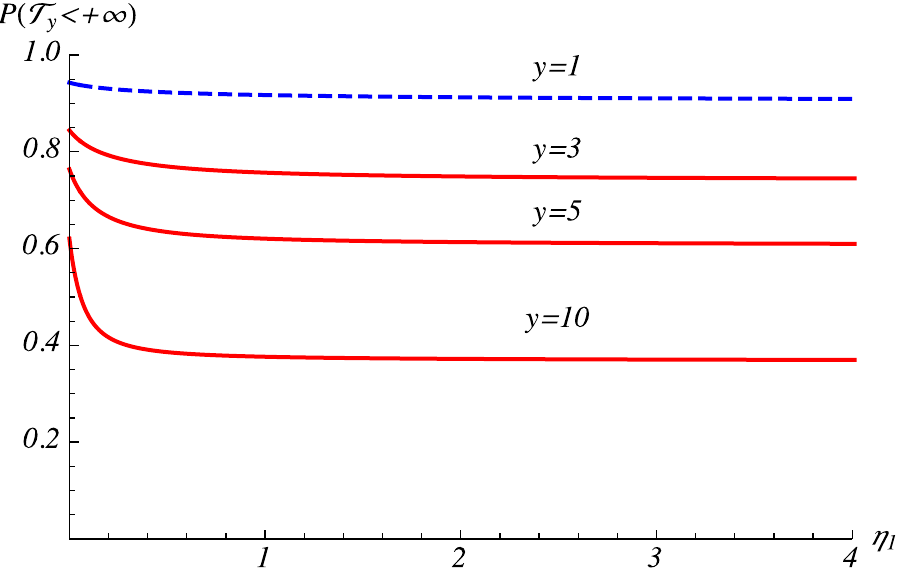}}\\
\caption{Plots of FPT probabilities $\mathbb{P}({\cal T}_y<+\infty)$ as function of $\eta_1$ for  $\eta_2=0$, $\lambda^*_2=1.2$, $\mu^*_2=1.0$, $\omega_1^2=1$,  $\omega_2^2=4$ 
and  $p=0.4$.}
\label{fig17}
\end{figure}
%
For $\lambda_2^*>\mu_2^*$, in Figure~\ref{fig17} we plot  $\mathbb{P}({\cal T}_y<+\infty)$, given in (\ref{eq:LTBPprob_cont}), as function of $\eta_1$ for $y=1,3,5,10$. 
\par
When $\lambda_2^*<\mu_2^*$, $\eta_1>0$ and $y\in\mathbb{R^+}$, from (\ref{eq:LTBPdensity_cont}) we obtain the FPT mean 
\begin{equation}
\mathbb{E}({\cal T}_y)={y\over \mu_2^*-\lambda_2^*}+{p\over\eta_1}\,\Bigl(1-{\mu_1^*-\lambda_1^*\over\mu_2^*-\lambda_2^*}\Bigr)\Bigl(1-e^{-y\zeta_1(0)}\Bigr).
\label{FPT_mean_contin}
\end{equation}
\par
Finally, in Figure~\ref{fig18} we plot the mean (\ref{FPT_mean_contin}) as function of $\eta_1$; since $\lambda_2^*< \mu_2^*$, the first passage through zero state is a 
certain event.
\begin{figure}[t]  
\centering
\subfigure[$\lambda_1^*=1.0,\mu_1^*=0.5$]{ \includegraphics[width=0.42\textwidth,keepaspectratio]{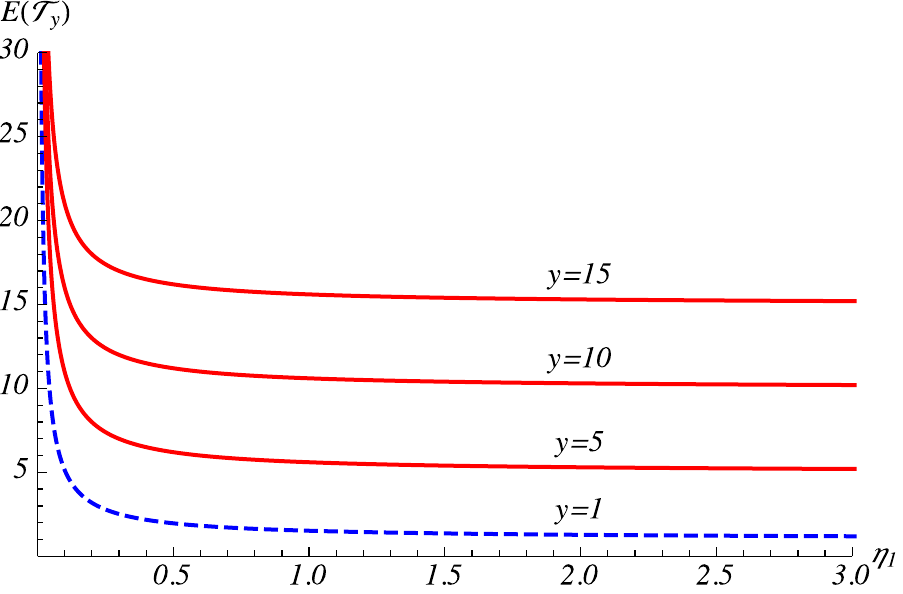}}
\hspace*{4mm}
\subfigure[$\lambda_1^*=0.5,\mu_1^*=1.0$]{ \includegraphics[width=0.42\textwidth,keepaspectratio]{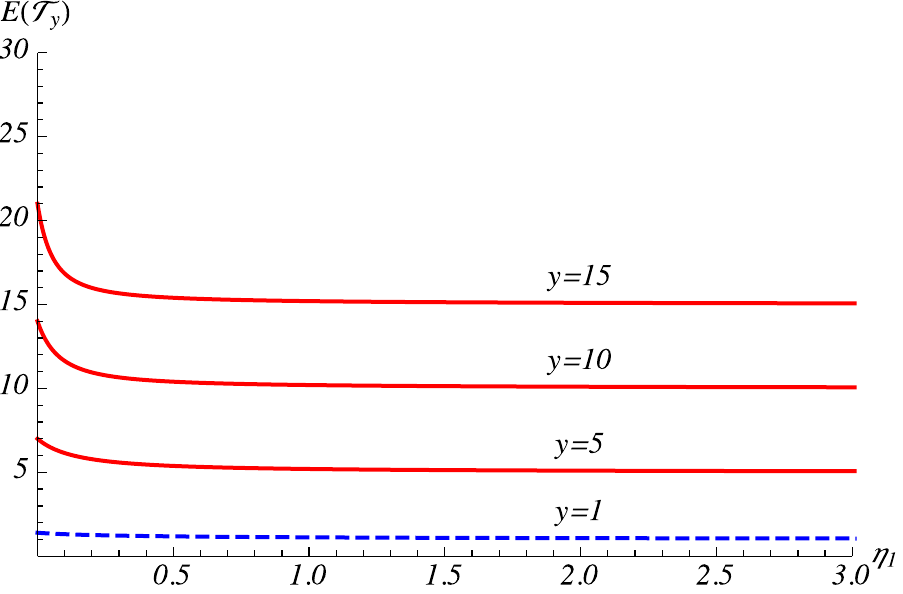}}\\
\caption{Plots of FPT mean $\mathbb{E}({\cal T}_y)$ as function of $\eta_1$ for  
$\eta_2=0$, $\lambda^*_2=1.0$, $\mu^*_2=2.0$, $\omega_1^2=1$,  $\omega_2^2=4$ and  $p=0.4$.}
\label{fig18}
\end{figure}
\section*{Concluding remarks} 
In this paper we considered a an $M/M/1$ queue whose behavior fluctuates randomly between 
two different environments according to a two-state continuous-time Markov chain. 
\par
We first get the steady-state distribution of the system, which is expressed via a generalized mixture 
of two geometric distributions. A remarkable result is that the system admits of a steady-state 
distribution even if one of the alternating environments does not possess a steady state. 
Hence, the switching between the environments can be used to stabilize a non stationary 
$M/M/1$ queue by means of the random alternation with a similar queue characterized by steady state. 
\par
Moreover, attention has been given to the transient distribution of the alternating queue, which can be 
expressed in a series form involving the queue-length distribution in absence of switching. 
A similar result is obtained also for the first-passage-time density through the zero state, in order 
to investigate the busy period. 
\par
The second part of the paper has been centered on a heavy-traffic approximation of the queue-length 
process, that leads to an alternating Wiener process restricted by a reflecting boundary at zero. 
The analysis of the approximating process  has been devoted to the steady-state density, 
which is expressed as a generalized mixture of two exponential densities. Moreover, we determined 
the transition density when only one type of switch is allowed.  
Such density can be decomposed in an integral form  involving the expressions of the  Wiener 
process in the presence of a reflecting boundary at zero. 
Finally, we analyzed the first-passage-time density through the zero state, which 
gives a suitable approximation of the busy period density. 
\section*{Acknowledgements} 
This research is partially supported by the group GNCS of INdAM.
%

%
\end{document}